\documentclass[11pt]{article}
\usepackage{latexsym,amsfonts,amssymb,amsmath,amsthm}

\usepackage{color}

\begin{document}

\newtheorem{tm}{Theorem}[section]
\newtheorem{prop}[tm]{Proposition}
\newtheorem{defin}[tm]{Definition}
\newtheorem{coro}[tm]{Corollary}

\newtheorem{lem}[tm]{Lemma}
\newtheorem{assumption}[tm]{Assumption}

\newtheorem{rk}[tm]{Remark}

\newtheorem{nota}[tm]{Notation}
\numberwithin{equation}{section}

\newcommand{\stk}[2]{\stackrel{#1}{#2}}
\newcommand{\dwn}[1]{{\scriptstyle #1}\downarrow}
\newcommand{\upa}[1]{{\scriptstyle #1}\uparrow}
\newcommand{\nea}[1]{{\scriptstyle #1}\nearrow}
\newcommand{\sea}[1]{\searrow {\scriptstyle #1}}
\newcommand{\csti}[3]{(#1+1) (#2)^{1/ (#1+1)} (#1)^{- #1
 / (#1+1)} (#3)^{ #1 / (#1 +1)}}
\newcommand{\RR}[1]{\mathbb{#1}}

\newcommand{ \bl}{\color{blue}}
\newcommand {\rd}{\color{red}}
\newcommand{ \bk}{\color{black}}
\newcommand{ \gr}{\color{OliveGreen}}
\newcommand{ \mg}{\color{RedViolet}}

\newcommand{\ep}{\varepsilon}
\newcommand{\rr}{{\mathbb R}}
\newcommand{\alert}[1]{\fbox{#1}}

\newcommand{\eqd}{\sim}
\def\R{{\mathbb R}}
\def\N{{\mathbb N}}
\def\Q{{\mathbb Q}}
\def\C{{\mathbb C}}
\def\l{{\langle}}
\def\r{\rangle}
\def\t{\tau}
\def\k{\kappa}
\def\a{\alpha}
\def\la{\lambda}
\def\De{\Delta}
\def\de{\delta}
\def\ga{\gamma}
\def\Ga{\Gamma}
\def\ep{\varepsilon}
\def\eps{\varepsilon}
\def\si{\sigma}
\def\Re {{\rm Re}\,}
\def\Im {{\rm Im}\,}
\def\E{{\mathbb E}}
\def\P{{\mathbb P}}
\def\Z{{\mathbb Z}}
\def\D{{\mathbb D}}
\newcommand{\ceil}[1]{\lceil{#1}\rceil}

\title{Existence of Traveling wave solutions of parabolic-parabolic chemotaxis systems}

\author{
Rachidi B. Salako and Wenxian Shen  \\
Department of Mathematics and Statistics\\
Auburn University\\
Auburn University, AL 36849\\
U.S.A. }

\date{}
\maketitle
\begin{abstract}
\noindent The current paper is devoted to the study of traveling wave solutions of the following parabolic-parabolic chemotaxis system,
\begin{equation*}\label{main-eq-abstract}
\begin{cases}
u_{t}= \Delta u-\chi  \nabla \cdot (u \nabla v) + u(a-bu),\quad x\in\mathbb{R}^N\\
\tau v_t=\Delta v-v+u, \quad x\in\mathbb{R}^N,
\end{cases}
\end{equation*}
where $u(x,t)$ represents the population density of a mobile species and $v(x,t)$ represents the population density of a  chemoattractant, and
$\chi$ represents the chemotaxis sensitivity.

 In  an earlier work (\cite{SaSh2})  by the authors of the current paper,  traveling wave solutions of the above chemotaxis system with $\tau=0$ are
studied. It is shown in \cite{SaSh2} that for every $0<\chi<\frac{b}{2}$,  there is $c^*(\chi)$ such that for every $c>c^*(\chi)$ and $\xi\in S^{N-1}$,  the system
 has a traveling wave solution $(u(x,t),v(x,t))=(U(x\cdot\xi-ct;\tau),V(x\cdot\xi-ct;\tau))$ with speed $c$ connecting the constant solutions $(\frac{a}{b},\frac{a}{b})$ and $(0,0)$. Moreover,
$$
\lim_{\chi\to 0+}c^{*}(\chi)=\begin{cases}
2\sqrt{a}  \qquad\ \text{if} \quad 0<a\leq 1\cr
1+a \qquad \text{if} \qquad a>1.
\end{cases}
$$

We prove in the current paper that for every $\tau >0$, there is $0<\chi_{\tau}^*<\frac{b}{2}$ such that for every $0<\chi<\chi_{\tau}^*$, there exist  two positive numbers $  c^{**}(\chi,\tau)>c^{*}(\chi,\tau)\ge 2\sqrt a$ satisfying that for every $ c\in   ( c^{*}(\chi,\tau)\ ,\ c^{**}(\chi,\tau))$ and $\xi\in S^{N-1}$, the system
 has a traveling wave solution $(u(x,t),v(x,t))=(U(x\cdot\xi-ct;\tau),V(x\cdot\xi-ct;\tau))$ with speed $c$ connecting the constant solutions $(\frac{a}{b},\frac{a}{b})$ and $(0,0)$, and it does not have such traveling wave solutions of speed less than $2\sqrt a$. Moreover,
$$
\lim_{\chi\to 0+}c^{**}(\chi,\tau)=\infty,$$
$$
\lim_{\chi\to 0+}c^{*}(\chi,\tau)=\begin{cases}
2\sqrt{a}\qquad \qquad  \qquad\ \text{if} \quad 0<a\leq \frac{1+\tau a}{(1-\tau)_+}\cr
\frac{1+\tau a}{(1-\tau)_{+}}+\frac{a(1-\tau)_{+}}{1+\tau a}\quad \text{if} \quad a\geq \frac{1+\tau a}{(1-\tau)_+},
\end{cases}
$$
and
$$
\lim_{x\to \infty}\frac{U(x;\tau)}{e^{-\mu x}}=1,
$$
where $\mu$ is the only solution of the equation $\mu+\frac{a}{\mu}=c$ in the interval $(0\ ,\ \min\{\sqrt{a}, \sqrt{\frac{1+\tau a}{(1-\tau)_+}}\})$. Furthermore,
 $$\lim_{\tau\to 0+}\chi_{\tau}^*=\frac{b}{2},\quad \lim_{\tau\to 0+} c^*(\chi;\tau)=c^*(\chi),\quad \lim_{\tau \to 0+}c^{**}(\chi;\tau)=\infty.
$$
\end{abstract}



\section{Introduction}
At  the beginning of 1970s, Keller and Segel  (see \cite{KeSe1},
\cite{KeSe2}) introduced systems of partial differential equations of the following form  to model the time
evolution of  the density $u(x,t)$ of a mobile species and the
density $v(x,t)$ of a  chemoattractant,
\begin{equation}\label{IntroEq0}
\begin{cases}
u_{t}=\nabla\cdot (m(u)\nabla u- \chi(u,v)\nabla v) + f(u,v),\quad   x\in\Omega \\
\tau v_t=\Delta v + g(u,v),\quad  x\in\Omega
\end{cases}
\end{equation}
complemented with certain boundary condition on $\partial\Omega$ if $\Omega$ is bounded, where $\Omega\subset \R^N$ is an open domain;  $\tau\ge 0$ is a non-negative constant linked to the speed of diffusion of the chemical;  the function $\chi(u,v)$ represents  the sensitivity with respect to chemotaxis; and the functions $f$ and $g$ model the growth of the mobile species and the chemoattractant, respectively.
In literature, \eqref{IntroEq0} is called the Keller-Segel model or a chemotaxis model.

Since the works by Keller and Segel,  a rich variety of mathematical models for studying chemotaxis has appeared (see \cite{BBTW}, \cite{DiNa}, \cite{DiNaRa}, \cite{GaSaTe}, \cite{KKAS},  \cite{NAGAI_SENBA_YOSHIDA}, \cite{Sug},  \cite{SuKu},  \cite{TeWi},   \cite{WaMuZh},  \cite{win_jde}, \cite{win_JMAA_veryweak}, \cite{win_arxiv}, \cite{Win}, \cite{win_JNLS}, \cite{YoYo}, \cite{ZhMuHuTi}, and the references therein). In the current paper and
the further coming papers, we will investigate various dynamical aspects of the following parabolic-parabolic Keller-Segel systems,
\begin{equation}\label{main-eq0}
\begin{cases}
u_{t}=\Delta u- \chi\nabla (u\nabla v) + u(a-bu),\quad  x\in\R^N \\
\tau v_t=\Delta v -v+u,\quad  x\in\R^N.
\end{cases}
\end{equation}

Note that, in the absence of the chemotaxis (i.e. $\chi=0$), the first equation of \eqref{main-eq0} becomes
\begin{equation}
\label{kpp-eq}
u_{t}=\Delta u+ u(a-bu),\quad  x\in\R^N,
\end{equation}
which is referred to as Fisher or KPP equation  due to  the pioneering works by Fisher (\cite{Fis}) and Kolmogorov, Petrowsky, Piscunov
(\cite{KPP}).  Among important solutions of \eqref{kpp-eq}
are traveling wave solutions of \eqref{kpp-eq} connecting the constant solutions $a/b$ and $0$.
It is well known that \eqref{kpp-eq} has traveling wave solutions $u(t,x)=\phi(x-ct)$ connecting $\frac{a}{b}$ and $0$ (i.e.
$(\phi(-\infty)=a/b$, $\phi(\infty)=0)$) for all speeds $c\geq 2\sqrt{a}$ and has no such traveling wave
solutions of slower speed (see \cite{Fis, KPP, Wei1}). Moreover, the stability of traveling wave solutions of \eqref{kpp-eq} connecting $\frac{a}{b}$ and $0$ has also been studied (see \cite{Bra}, \cite{Sat}, \cite{Uch}, etc.).
The above mentioned results for \eqref{kpp-eq} have also been
well extended to   reaction diffusion equations of the form,
\begin{equation}
\label{general-kpp-eq}
u_t=\Delta u+u f(t,x,u),\quad x\in\R^N,
\end{equation}
where $f(t,x,u)<0$ for $u\gg 1$,  $\partial_u f(t,x,u)<0$ for $u\ge 0$ (see \cite{Berestycki1, BeHaNa1, BeHaNa2, Henri1, Fre, FrGa, LiZh, LiZh1, Nad, NoRuXi, NoXi1, She1, She2, Wei1, Wei2, Zla}, etc.).

Similar to \eqref{kpp-eq}, traveling wave solutions  connecting the constant solutions $(a/b,a/b)$ and $(0,0)$ are among most important solutions of \eqref{main-eq0}. However, such solutions have been hardly studied. The objective of the current paper is to study
 the existence of traveling wave solutions connecting $(a/b,a/b)$ and $(0,0)$.
A nonnegative solution $(u(x,t),v(x,t))$ of \eqref{main-eq0} is called a {\it traveling wave solution} connecting $(a/b,a/b)$ and $(0,0)$ and propagating
in the direction $\xi\in S^{N-1}$ with speed $c$ if it is of the form
$(u(x,t),v(x,t))=(U(x\cdot\xi-ct),V(x\cdot\xi-ct))$ with
$\lim_{z\to -\infty}(U(z),V(z))=(a/b,a/b)$ and $\lim_{z\to\infty}(U(z),V(z))=(0,0)$.

Observe that,  if $(u(x,t),v(x,t))=(U(x\cdot\xi-ct),V(x\cdot\xi-ct))$  $(x\in\R^N, \ t\ge 0)$ is a traveling wave solution of \eqref{main-eq0} connecting  $(a/b,a/b)$ and $(0,0)$ and propagating
in the direction $\xi\in S^{N-1}$, then $(u,v)=(U(x-ct),V(x-ct))$ ($x\in\R$)
is a traveling wave solution of
\begin{equation}
\label{main-eq1}
\begin{cases}
u_{t}=u_{xx} -\chi (u v_x)_x + u(a-bu),\quad x\in\R\\
\tau v_t=v_{xx}-v+u, \quad x\in\R
\end{cases}
\end{equation}
connecting  $(a/b,a/b)$ and $(0,0)$. Conversely, if $(u(x,t),v(x,t))=(U(x-ct),V(x-ct))$ ($x\in\R,\ t\geq 0$) is a traveling wave solution
of \eqref{main-eq1} connecting $(a/b,a/b)$ and $(0,0)$, then $(u,v)=(U(x\cdot\xi-ct),V(x\cdot\xi-ct))$  $(x\in\R^N)$ is a traveling wave solution of \eqref{main-eq0} connecting  $(a/b,a/b)$ and $(0,0)$ and propagating
in the direction $\xi\in S^{N-1}$. In the following, we will then study the existence of traveling wave solutions
of \eqref{main-eq1} connecting $(a/b,a/b)$ and $(0,0)$.

Observe also that,  $(u,v)=(U(x-ct),V(x-ct))$  is a traveling wave solution of \eqref{main-eq1} connecting $(a/b,a/b)$ and $(0,0)$  if and only if $(u,v)=(U(x),V(x))$ is a stationary solution of the following parabolic-parabolic chemotaxis system,
\begin{equation}
\label{reduced-eq1}
\begin{cases}
u_{t}=u_{xx}+cu_x -\chi (u v_x)_x + u(a-bu),\quad x\in\R\\
\tau v_t=v_{xx}+\tau c v_x-v+u, \quad x\in\R,
\end{cases}
\end{equation}
and $(u,v)=(U(x),V(x))$ is also a stationary solution of the following parabolic-elliptic chemotaxis system,
\begin{equation}
\label{reduced-eq2}
\begin{cases}
u_{t}=u_{xx}+cu_x -\chi (u v_x)_x + u(a-bu),\quad x\in\R\\
0=v_{xx}+\tau c v_x-v+u, \quad x\in\R.
\end{cases}
\end{equation}
Clearly, \eqref{reduced-eq1} and \eqref{reduced-eq2} have same stationary solutions. In this paper, to study the existence of traveling
wave solutions of \eqref{main-eq1}, we study the existence of constant $c$'s so that \eqref{reduced-eq2} has a stationary solution $(U(x),V(x))$ with
$(U(-\infty),V(-\infty))=(a/b,a/b)$ and $(U(\infty),V(\infty))=(0,0)$.

To this end, we first prove the following two theorems on the global
existence of classical solutions of \eqref{reduced-eq2} and the
stability of constant solution $(a/b,a/b)$, which are of independent
interest.  Let
$$
C_{\rm unif}^b(\R)=\{u\in C(\R)\,|\, u(x)\quad \text{is uniformly continuous in}\,\,\, x\in\R\quad \text{and}\,\, \sup_{x\in\R}|u(x)|<\infty\}
$$
equipped with the norm $\|u\|_\infty=\sup_{x\in\R}|u(x)|$. For any $0\le \nu<1$, let
$$
C_{\rm unif}^{b,\nu}=\{u\in C_{\rm unif}^b\,|\, \sup_{x,y\in\R,x\not = y}\frac{|u(x)-u(y)|}{|x-y|^\nu}<\infty\}
$$
with norm $\|u\|_{C_{\rm unif}^{b,\nu}}=\sup_{x\in\R}|u(x)|+\sup_{x,y\in\R,x\not =y}\frac{|u(x)-u(y)|}{|x-y|^\nu}$.

Note that, for fixed $c$, it can be proved by the similar
arguments as those in \cite{SaSh} that for any $u_0\in C_{\rm
unif}^b(\R)$ with $u_0\ge 0$, there is $T_{\max}(u_0)\in (0,\infty]$
such that \eqref{reduced-eq2} has a unique classical solution
$(u(x,t;u_0),v(x,t;u_0))$ on $[0,T_{\max}(u_0))$ with
$u(x,0;u_0)=u_0(x)$. Furthermore, if $T_{\max}(u_0)<\infty$, then $\lim_{t\to T_{\max}(u_0)^-}\|u(\cdot,t;u_0)\|_{\infty}=\infty$.

\medskip
\noindent{\bf Theorem A.}
{\it Assume that $0\le \frac{\chi \tau c}{2}{ < }{b-\chi}$. Then
 for any  $u_0\in C_{\rm unif}^b(\R)$ with $0\le u_0$,  $T_{\max}(u_0)=\infty$.  Moreover,  the solution $(u(\cdot,\cdot;u_0),v(\cdot,\cdot;v_0))$ satisfies that
 \begin{equation}\label{global-exixst-thm-eq1}
 \|u(\cdot,t;u_0)\|_{\infty}\leq \max\{\|u_0\|_{\infty},\frac{a}{b-\chi-\frac{\chi c\tau}{2}}\}\end{equation}\quad \text{and}\quad\begin{equation}\label{global-exixst-thm-eq2}
\|v(\cdot,t;u_0)\|_{\infty}\leq\max\{\|u_0\|_{\infty}\ ,\ \frac{a}{b-\chi -\frac{\chi c\tau}{2}}\}
 \end{equation}
 for every $t\geq 0$.
}

\medskip

\noindent {\bf Theorem B.}
{\it Assume that { $0\leq \chi\tau c<b-2\chi$}. Then for any $u_0\in C_{\rm unif}^b(\R)$ with $\inf_{x\in\R}u_0(x)>0$,
$$
\lim_{t\to\infty}\Big[\|u(\cdot,t;u_0)-\frac{a}{b}\|_\infty+\|v(\cdot,t;u_0)-\frac{a}{b}\|_\infty\Big]=0.
$$
}

\medskip

\noindent {\bf Remark 1.} Theorem A in the case $c=0$ recovers \cite[Theorem 1.5]{SaSh} in the case $b>\chi$
and Theorem B in the case $c=0$ recovers \cite[Theorem 1.8]{SaSh}.

\medskip

 Next, we consider the existence and nonexistence of traveling
wave solutions of \eqref{main-eq1}. Note that, when $\tau=0$, the
following result is proved in \cite{SaSh2}.

\begin{itemize}
\item[$\bullet$] {\it  For any given $0<\chi<\frac{b}{2}$, let
$\mu^*(\chi)$ be defined by
$$
\mu^*(\chi)=\sup\{\mu\,|\, 0<\mu<\min\{1,\sqrt a\},\,\,\,  \frac{\mu (\mu+\sqrt {1-\mu^2})}{1-\mu^2}\le \frac{b-\chi}{\chi}\}.
$$
Let
$$
c^*(\chi)=\mu^*(\chi)+\frac{a}{\mu^*(\chi)}.
$$
Then for any $c> c^*(\chi)$, \eqref{main-eq1} with $\tau=0$ has a traveling wave solution $(u,v)=(U(x-ct;\tau),V(x-ct;\tau))$ with speed $c$ connecting the constant solutions $(\frac{a}{b},\frac{a}{b})$ and $(0,0)$. Moreover,
$$
\lim_{\chi\to 0+}c^{*}(\chi)=\begin{cases}
2\sqrt{a}\qquad \qquad  \qquad\ \text{if} \quad 0<a\leq 1\cr
1+a\qquad\qquad\qquad \text{if} \quad a>1.
\end{cases}
$$
}
\end{itemize}

In this paper, we prove the following theorems on the existence and nonexistence of traveling wave solutions of \eqref{main-eq1} with $\tau>0$.

\medskip

\noindent{\bf Theorem C.}
{\it  For every $\tau >0$, there is $0<\chi_{\tau}^*<\frac{b}{2}$ such that for every $0<\chi<\chi_{\tau}^*$, there exist  two positive numbers $0< c^{*}(\chi,\tau)<c^{**}(\chi,\tau)$ satisfying that for every $ c\in   ( c^{*}(\chi,\tau)\ ,\ c^{**}(\chi,\tau))$, \eqref{main-eq1} has a traveling wave solution $(u,v)=(U(x-ct;\tau),V(x-ct;\tau))$ with speed $c$ connecting the constant solutions $(\frac{a}{b},\frac{a}{b})$ and $(0,0)$. Moreover,
$$
\lim_{\chi\to 0+}c^{**}(\chi,\tau)=\infty,$$
$$
\lim_{\chi\to 0+}c^{*}(\chi,\tau)=\begin{cases}
2\sqrt{a}\qquad \qquad  \qquad\ \text{if} \quad 0<a\leq \frac{1+\tau a}{(1-\tau)_+}\cr
\frac{1+\tau a}{(1-\tau)_{+}}+\frac{a(1-\tau)_{+}}{1+\tau a}\quad \text{if} \quad a\geq \frac{1+\tau a}{(1-\tau)_+},
\end{cases}
$$
and
$$
\lim_{x\to \infty}\frac{U(x;\tau)}{e^{-\mu x}}=1,
$$
where $\mu$ is the only solution of the equation $\mu+\frac{a}{\mu}=c$ in the interval $(0\ ,\ \min\{\sqrt{a}, \sqrt{\frac{1+\tau a}{(1-\tau)_+}}\})$.  Furthermore,
$$ \lim_{\tau\to 0+}\chi_{\tau}^*=\frac{b}{2},
$$
and
$$\lim_{\tau \to 0+}c^{*}(\chi,\tau)=c^*(\chi),\quad \lim_{\tau\to 0+}c^{**}(\chi,\tau)=\infty $$
for every $0<\chi<\frac{b}{2}$.}

\medskip

\noindent {\bf Theorem D.} {\it
For any given $\tau\ge 0$ and $\chi\ge 0$,   \eqref{main-eq1} has no traveling wave solution $(u,v)=(U(x-ct;\tau),V(x-ct;\tau))$
with $(U(-\infty),V(-\infty))=(\frac{a}{b},\frac{a}{b})$, $(U(\infty),V(\infty))=(0,0)$, and $c<2 \sqrt a$.
}

\medskip

\noindent{\bf Remark 2.}
\begin{description}

\item[(i)]  By Theorem C, when $0<a\leq \frac{1+\tau a}{(1-\tau)_+}$,  then, as $\chi\to 0+$, $c^{*}(\chi;\tau)$ converges to the minimal speed
 $2\sqrt a$ of \eqref{kpp-eq}. Moreover, if $0<a< \frac{1+\tau a}{(1-\tau)_+}$  and
 $$ 0<\chi<\frac{b}{1+\max\{1+2\tau\sqrt{a}\ ,\ \frac{\sqrt{a}(1+2\tau)}{\sqrt{1+2\tau a-a}}+\frac{a(1+2\tau)}{1+2\tau z-a} \}},$$
    then $c^*(\chi;\tau)=2\sqrt a$ (see Remark \ref{rk-thmc}).

 \item[(ii)]  For given  $\tau>0$ and $\chi\ge 0$,   let $(c_{\min}^*(\chi,\tau), c_{\max}^*(\chi,\tau))$ be the largest subinterval of $(0,\infty)$ such that
  $(c^{*}(\chi,\tau),c^{**}(\chi;\tau))\subset    (c_{\min}^*(\chi,\tau),c_{\max}^*(\chi,\tau))$ and  for   any  $c\in (c_{\min}^*(\chi,\tau),c_{\max}^*(\chi,\tau))$, \eqref{main-eq1} has a traveling wave
 solution connecting $(a/b,a/b)$ and $(0,0)$ with speed $c$.
 By Theorems C and D, for $0<\chi<\chi_\tau^*$,
 $$
 2\sqrt a\le c_{\min}^*(\chi,\tau)\le c^*(\chi,\tau)<c^{**}(\chi,\tau)\le c_{\max}^*(\chi,\tau).
 $$
 It   remains also open whether $c_{\min}^*(\chi,\tau)=2\sqrt a$ and whether $c_{\max}^*(\chi,\tau)=\infty$. The first question is about whether
 the chemotaxis increases the minimal wave speed and the second question is about whether the chemotaxis prevents the existence of traveling wave
 solutions with large speeds. It is of great theoretical and biological interests to investigate these two questions.
\end{description}

 Because of the lack of comparison principle,  the proofs of Theorems C is highly non trivial. Note that very few results are known about the dynamics of solutions of \eqref{main-eq0} when $a>0,\ b>0$ and $ \tau >0$ for initial data which are bounded and uniformly continuous.  To our best knowledge, there is no existing results on the stability of the trivial solution $(\frac{a}{b},\frac{a}{b})$ of \eqref{main-eq0} when $\tau>0$. This makes the study of traveling wave solutions of \eqref{main-eq0} more complicated.

 Our first key step toward the proof Theorem C is the relationship between traveling wave solutions of \eqref{main-eq1} and stationary solutions of  \eqref{reduced-eq2}. We have that $(u,v)=(U(x-ct),V(x-ct))$ is traveling wave solution of \eqref{main-eq1} connecting $(a/b,a/b)$ and $(0,0)$
  if and only if $(u,v)=(U(x),V(x))$ is a stationary solution of \eqref{reduced-eq2} connecting $(a/b,a/b)$ and $(0,0)$. This  observation leads to the study of the dynamics of solutions to \eqref{reduced-eq2}. Concerning \eqref{reduced-eq2}, we first establish Theorem A, which provides, for given $c$, sufficient conditions for the existence global classical solutions.  Next, we establish Theorem B, which provides, for given $c$,  sufficient conditions for the stability of the constant solution $(\frac{a}{b},\frac{a}{b})$ with respect to positive perturbations for \eqref{reduced-eq2}.

 The proof of Theorem C involves the proof of the existence of $c$ such that   \eqref{reduced-eq2} has stationary solutions connecting $(a/b,a/b)$ and $(0,0)$. The proof of which is based on the construction of a bounded convex non-empty subset of $C_{\rm unif}^{b}(\R)$, called $\mathcal{E}_{\mu,\tau}(C_0)$ (see \eqref{definition-E-mu}), and a continuous and compact function  $U : \mathcal{E}_{\mu,\tau}(C_0)\to \mathcal{E}_{\tau,\mu}(C_0)$. Any fixed point of this function, whose existence is guaranteed by the Schauder's fixed theorem,  becomes a stationary solution of  \eqref{reduced-eq2} connecting $(a/b,a/b)$ and $(0,0)$.
 The construction of the set $\mathcal{E}_{\tau,\mu}(C_0)$ itself is also based on the construction of two special functions. These two special functions are sub-solution and sup-solution of a collection of parabolic equations.
The proof of Theorem D utilizes some principal eigenvalue theory for elliptic equations.

In our future works, we plan to study local/global existence of classical solutions, asymptotic behaviors,  and spatial spreading speeds of classical solutions
of \eqref{main-eq0} with $\tau>0$ for nonnegative, bounded and uniformly continuous initial. These questions are very important in the understanding of dynamics of \eqref{main-eq0}.
When $\tau=0$, we refer the reader to \cite{SaSh} and the references therein.

It should be mentioned that there are  several studies of \eqref{main-eq0} on bounded domains (see \cite{Win} and the references therein). When $\Omega$ is
 a convex bounded domain, it is known that when the logistic damping coefficient is large enough, then the constant solution $(\frac{a}{b},\frac{a}{b})$ is stable (see \cite{Win}). It should be also  pointed out that there are many studies on
traveling wave solutions of several other types of chemotaxis
models, see, for example, \cite{AiHuWa, AiWa, FuMiTs, HoSt, LiLiWa,
MaNoSh, NaPeRy,Wan}, etc. In particular, the reader is referred to
the review paper \cite{Wan}.

The rest of this paper is organized as follows. Section 2 is to study  the dynamics of the induced parabolic-elliptic chemotaxis system \eqref{reduced-eq2} and prove Theorems A and B. Section 3 is to establish  the tools that will be needed in the proof of Theorem C. It is here that we define the two special functions, which are sub-solution and sup-solution of a collection of parabolic equations,  and the non-empty bounded and convex subset $\mathcal{E}_{\tau,\mu}(C_0)$. In section 4, we study  the existence and nonexistence of traveling wave solutions and prove Theorems C and D.

\section{Dynamics of the  induced parabolic-elliptic chemotaxis system}

In this section, we study the global existence of classical solutions of \eqref{reduced-eq2} with given nonnegative initial functions
and the stability of the constant solution $(a/b,a/b)$ of \eqref{reduced-eq2}, and prove Theorems A and B.


For fixed $c$, it can be proved by the similar arguments as those in \cite{SaSh} that for any $u_0\in C_{\rm unif}^b(\R)$ with $u_0\ge 0$, there is
$T_{\max}(u_0)\in (0,\infty]$ such that \eqref{reduced-eq2}
has a unique classical solution $(u(x,t;u_0),v(x,t;u_0))$ on $[0,T_{\max}(u_0))$ with
$u(x,0;u_0)=u_0(x)$.

 Observe that, \eqref{reduced-eq2} is equivalent to
\begin{equation}
\label{stationary-eq}
\begin{cases}
u_{t}=u_{xx}+(c -\chi v_{x})u_{x} + u(a-\chi( v-\tau cv_{x})-(b-\chi)u),\quad x\in\R\\
0=v_{xx}+\tau cv_{x}-v+u, \quad x\in\R.
\end{cases}
\end{equation}

\begin{proof}[Proof of Theorem A]
For given $u_0\in C_{\rm unif}^b(\R)$ and $T>0$, choose $C_0>0$ such that $0\le u_0\le C_0$ and $C_0\ge \frac{a}{b-\chi-\frac{\chi\tau c}{2}}$,  and  let
$$
\mathcal{E}(u_0,T):=\{u\in C_{\rm unif}^b(\R\times [0,T])\,|\, u(x,0)=u_0(x),\,\, 0\le u(x,t)\le  C_0\}.
$$
For given $u\in \mathcal{E}(u_0,T)$, let $v(x,t;u)$ be the solution of the second equation in \eqref{stationary-eq}.
Then for every $u\in\mathcal{E}(u_0,T)$,  $x\in\R, t\geq 0$, we have that
\begin{equation}\label{v-eq0001}
v(x,t;u)=\int_{0}^{\infty}\int_{\R}\frac{e^{-s}}{\sqrt{4\pi s}}e^{-\frac{|x-z|^2}{4s}}u(z+\tau cs,t)dzds,
\end{equation}
and
\begin{align}\label{partial-v-eq0001}
v_{x}(x,t;u_0)&=\frac{1}{\sqrt{\pi}}\int_0^\infty\int_\R\frac{(z-x)e^{-s}}{2s\sqrt{4\pi s}}e^{-\frac{|x-z|^2}{4s}}u(x+\tau cs,t)dzds\nonumber\\
&=\frac{1}{\sqrt{\pi}}\int_{0}^{\infty}\int_{\R}\frac{ze^{-s}}{\sqrt{s}}e^{-z^2}u(x+2\sqrt{s}z+\tau cs,t)dzds\nonumber\\
 & =\frac{1}{\sqrt{\pi}}\int_{0}^{\infty}\int_{0}^{\infty}\frac{ze^{-s}}{\sqrt{s}}e^{-z^2}u(x+2\sqrt{s}z+\tau cs,t)dzds\nonumber\\
&  - \frac{1}{\sqrt{\pi}}\int_{0}^{\infty}\int_{0}^{\infty}\frac{ze^{-s}}{\sqrt{s}}e^{-z^2}u(x-2\sqrt{s}z+\tau cs,t)dzds .
\end{align}
Using the fact that
$
\int_{0}^{\infty}\frac{e^{-s}}{\sqrt{s}}ds=\sqrt{\pi}$ and $\int_{0}^{\infty}ze^{-z^2}dz=\frac{1}{2}$ it follows from \eqref{v-eq0001} and \eqref{partial-v-eq0001} that
\begin{equation}\label{Estimate-v-partial-v-eq0001}
|v(x,t;u)|\le C_0,\quad |v_x(x,t;u)|\le \frac{C_0}{2}\quad \forall\, t\in [0,T],\,\, x\in\R.
\end{equation}

For given $u\in\mathcal{E}(u_0,T)$, let $\tilde U(x,t;u)$ be the solution of the initial value problem
\begin{equation}
\label{aux-eq1}
\begin{cases}\tilde U_t(x,t)= \tilde U_{xx}+(c-\chi v_x(x,t;u))\tilde U_x\cr
\qquad\qquad\,\,  + \tilde U
\big((a-\chi( v(x,t;u)-\tau cv_{x}(x,t;u))-(b-\chi)\tilde U\big),\quad x\in\R\cr
\tilde U(x,0;u)=u_0(x), \quad x\in\R.
\end{cases}
\end{equation}
Since $u_0\ge 0$, comparison principle for parabolic equations implies that $\tilde U(x,t)\ge 0$ for every $x\in\R, \ t\in[0, T]$. Using the fact that $v(\cdot,\cdot,u)\geq 0$, it follows from \eqref{Estimate-v-partial-v-eq0001} that
$$
\tilde U_t(x,t)\leq \tilde U_{xx}+(c-\chi v_x(x,t;u))\tilde U_x+ \tilde U\big((a+ \frac{\chi\tau cC_0}{2})-(b-\chi)\tilde U\big).
$$
Hence, comparison principle for parabolic equations implies that
\begin{equation}\label{auxx-eq1}
\tilde U(x,t,u)\leq \tilde u(t,\|u_0\|_\infty), \quad \forall \ x\in\R,\ t\in[0, T],
\end{equation}
 where $\tilde u$ is the solution of the ODE
$$
\begin{cases}
\tilde u_t= \tilde u \big(a+ \frac{\chi \tau c C_0}{2}-(b-\chi)\tilde u)\quad t>0\cr
\tilde u(0)=\|u_0\|_\infty.
\end{cases}
$$
Since $b-\chi>0$, the function $\tilde  u(\cdot,\|u_0\|_\infty)$ is defined for all time and satisfies $0\leq \tilde u(t,\|u_0\|_\infty)\leq \max\{\|u_0\|_{\infty},\ \frac{a+ \frac{\chi \tau c C_0}{2}}{b-\chi}\}$ for every $t\geq 0$. This combined with \eqref{auxx-eq1} yield that
\begin{equation}\label{auxx-eq2}
0\le \tilde U(x,t;u)\le \max\{\|u_0\|,  \frac{a+ \frac{\chi \tau c C_0}{2}}{b-\chi}\}\leq C_0\quad \forall\,\, x\in\R,\,\,\forall\,\, t\in [0,T].
\end{equation}
Note that the second inequality in \eqref{auxx-eq2} follows from the fact $\frac{a+ \frac{\chi \tau c C_0}{2}}{b-\chi}\leq C_0$ whenever $C_0\ge \frac{a}{b-\chi- \frac{\chi \tau c}{2}}$.  It then follows that  $\tilde U(\cdot,\cdot;u)\in\mathcal{E}(u_0,T)$.

\smallskip

Following the proof of Lemma 4.3 in \cite{SaSh2}, we can prove that the mapping $\mathcal{E}(u_0,T)\ni u\mapsto \tilde U(\cdot,\cdot;u)\in\mathcal{E}(u_0,T)$ has a fixed point
$\tilde U(x,t;u)=u(x,t)$. Note that if $\tilde U(\cdot,\cdot,u)=u$, then $(u(\cdot,\cdot),v(\cdot,\cdot;u))$ is a solution of \eqref{reduced-eq2}. Since $u(\cdot,\cdot;u_0)$ is the only solution of \eqref{reduced-eq2}, thus $u(\cdot,\cdot,u_0)=u(\cdot,\cdot)$. Hence, it follows from \eqref{auxx-eq2} that for any $T>0$,
\begin{equation}\label{u-eq0001}
0\le u(x,t;u_0)\le \max\{\|u_0\|,  \frac{a+ \frac{\chi \tau c C_0}{2}}{b-\chi}\}\quad \forall\,\, t\in [0,T].
\end{equation}
 This implies that $T_{\max}(u_0)=\infty$. Inequalities \eqref{global-exixst-thm-eq1} and \eqref{global-exixst-thm-eq2} follow from \eqref{u-eq0001} with $C_{0}=\max\{\|u_0\|_{\infty},\frac{a}{b-\chi- \frac{\chi \tau c }{2}}\}$.
\end{proof}

\medskip

In the next result, we prove the stability of the positive constant solution $(a/b,a/b)$ of \eqref{reduced-eq2}.

\medskip

\begin{proof}[Proof of Theorem B]
Let $u_{0}\in C^{b}_{\rm uinf}(\R)$ with $\inf_{x\in\R}u_0(x)>0$ be given. By Theorem A,
 $ \sup_{x\in\R,\ t\geq 0}u(x,t;u_0)<\infty$. Hence
$$
M:=\sup_{x\in\R,\ t\ge 0}(v(x,t,u_0)+|v_{x}(x,t;u_0)|)<\infty.
$$
Thus, we have that
$$
u_{t}\geq u_{xx}+(c-\chi v_{x})u_x +(a-\chi(1+\tau c)M-(b-\chi)u), \quad \forall x\in\R,\ t>0.
$$
Therefore, comparison principle for parabolic equations implies that
$$
u(x,t;u_{0})\geq U^{l}(t),\quad \forall\ x\in\R, \forall\ t\geq 0,
$$
where $U^{l}(t)$ is the solution of the ODE
$$
\begin{cases}
U_{t}=U(a-\chi(1+\tau c)M-(b-\chi)U)\cr
U(0)=\inf_{x}u_0(x).
\end{cases}
$$
Since $\inf_{x}u_0(x)>0$, we have that $U^{l}(t)>0$ for every $t\geq 0$. Thus
\begin{equation}\label{e00}
0<U^{l}(t)\leq \inf_{x\in\R}u(x,t;u_0),\quad \forall\,\,  t\geq 0.
\end{equation}

 Next, let us define
$$
\bar u=\limsup_{t\to\infty}\sup_{x\in\R} u(x,t;u_0),\quad \underline u=\liminf_{t\to\infty}\inf_{x\in\R} u(x,t;u_0).
$$
Then for any $\epsilon>0$, there is $T_\epsilon>0$ such that
\begin{equation}\label{e000}
\underline u-\epsilon<u(x,t;u_0)\le \bar u+\epsilon\quad \forall t\ge T_\epsilon,\quad x\in\R.
\end{equation}
Note that for every $x\in\R$ and $t\geq 0$ we have that
\begin{equation}\label{e0}
v(x,t;u_0)=\int_{0}^\infty\int_{\R}\frac{e^{-s}}{\sqrt{4\pi s}}e^{-\frac{|x-z|^2}{4s}}u(z+\tau cs,t;u_0)dzds.
\end{equation}
Thus it follows from \eqref{e000} and \eqref{e0} that
\begin{equation}\label{e1}
\underline{u}-\varepsilon\le v(x,t,u_0)\leq \overline{u}+\varepsilon \quad \text{and}\quad \ \forall x\in\R, \, t\ge T_\varepsilon.
\end{equation}
Using again \eqref{e0}, for every $x\in\R$ and every $t>0$, we have that
\begin{eqnarray}\label{e2}
v_{x}(x,t;u_0)&= &\frac{1}{\sqrt{\pi}}\int_{0}^{\infty}\int_{R}\frac{z}{\sqrt{s}}e^{-s}e^{-z^2}u(x+2\sqrt{s}z+\tau cs,t;u_0)dzds\nonumber\\
& = & -\frac{1}{\sqrt{\pi}}\int_{0}^{\infty}\int_{0}^{\infty}\frac{z}{\sqrt{s}}e^{-s}e^{-z^2}u(x-2\sqrt{s}z+\tau cs,t;u_0)dzds\nonumber\\
& &+\frac{1}{\sqrt{\pi}}\int_{0}^{\infty}\int_{0}^{\infty}\frac{z}{\sqrt{s}}e^{-s}e^{-z^2}u(x+2\sqrt{s}z+\tau cs,t,u_0)dzds.\nonumber\\
\end{eqnarray}
Combining \eqref{e000} and \eqref{e2}, for every $x\in\R$ and $t\geq T_{\varepsilon}$, we obtain that
\begin{eqnarray}\label{e3}
v_{x}(x,t,u_0)& \geq & -\frac{(\overline{u}
+\varepsilon)}{\sqrt{\pi}}\int_{0}^{\infty}\frac{e^{-s}}{\sqrt{s}}\Big[\int_{0}^{\infty}ze^{-z^2}dz\Big]ds+\nonumber\\
& & +\frac{(\underline{u}-
\varepsilon)}{\sqrt{\pi}}\int_{0}^{\infty}\frac{e^{-s}}{\sqrt{s}}\Big[\int_{0}^{\infty}ze^{-z^2}dz\Big]ds\nonumber\\
& = & -\frac{(\overline{u}-\underline{u}+2\varepsilon)}{2}
\end{eqnarray}
Similarly, combining \eqref{e000} and \eqref{e2}, for every $x\in\R$ and $t\geq T_{\varepsilon}$, we obtain that
\begin{eqnarray}\label{e4}
v_{x}(x,t,u_0)& \leq & -\frac{(\underline{u}
-\varepsilon)}{\sqrt{\pi}}\int_{0}^{\infty}\frac{e^{-s}}{\sqrt{s}}\Big[\int_{0}^{\infty}ze^{-z^2}dz\Big]ds+\nonumber\\
& & +\frac{(\overline{u}+
\varepsilon)}{\sqrt{\pi}}\int_{0}^{\infty}\frac{e^{-s}}{\sqrt{s}}\Big[\int_{0}^{\infty}ze^{-z^2}dz\Big]ds\nonumber\\
& = & \frac{(\overline{u}-\underline{u}+2\varepsilon)}{2}.
\end{eqnarray}
Thus, using inequalities \eqref{e1} and \eqref{e4}, we have that
$$
u_t\leq u_{xx}+(c-\chi v_{x})u_x+(a-\chi(\underline{u}-\varepsilon)+\frac{\chi\tau c}{2}(\overline{u}-\underline{u}+2\varepsilon)-(b-\chi)u ), \quad t\geq T_{\varepsilon}.
$$
Thus, comparison principle for parabolic equations implies that
\begin{equation}\label{e5}
u(x,t;u_0)\leq \overline{U}^{\varepsilon}(t),\quad \forall\, x\in\R,\ \forall\, t\geq T_{\varepsilon},
\end{equation}
where $\overline{U}^{\varepsilon}(t)$ is the solution of the ODE
$$
\begin{cases}
\overline{U}_{t}=\overline{U}((a-\chi(\underline{u}-\varepsilon)+\frac{\chi\tau c}{2}(\overline{u}-\underline{u}+2\varepsilon))_{+}-(b-\chi)\overline{U})\quad t> T_{\varepsilon}\cr
\overline{U}(T_{\varepsilon})=\|u(\cdot,T_{\varepsilon};u_0)\|_{\infty}.
\end{cases}
$$
Since $\|u(\cdot,T_{\varepsilon};u_0)\|_{\infty}>0$ and $b-\chi>0$, we have that $\overline{U}^{\varepsilon}(t)$ is defined for all time and satisfies $ \lim_{t\to\infty}\overline{U}^{\varepsilon}(t)= \frac{(a-\chi(\underline{u}-\varepsilon)+\frac{\chi\tau c}{2}(\overline{u}-\underline{u}+2\varepsilon))_{+}}{b-\chi}$. Thus, it follows from \eqref{e5} that
$$
\overline{u}\leq \frac{(a-\chi(\underline{u}-\varepsilon)+\frac{\chi\tau c}{2}(\overline{u}-\underline{u}+2\varepsilon))_{+}}{b-\chi},\quad \forall\, \varepsilon>0.
$$
Letting $\varepsilon$ tends to 0 in the last inequality, we obtain that
\begin{equation*}
\overline{u}\leq \frac{(a-\chi\underline{u}+\frac{\chi\tau c}{2}(\overline{u}-\underline{u}))_{+}}{b-\chi}
\end{equation*}
But, note that $ (a-\chi\underline{u}+\frac{\chi\tau c}{2}(\overline{u}-\underline{u}))_{+}=0$ would implies that $\underline{u}=\overline{u}=0$. Which in turn yields that
$$
0=(a-\chi\underline{u}+\frac{\chi\tau c}{2}(\overline{u}-\underline{u}))_{+}=a.
$$
This is impossible since $a>0$. Hence $(a-\chi\underline{u}+\frac{\chi\tau c}{2}(\overline{u}-\underline{u}))_{+}>0$. Whence
\begin{equation}\label{e6}
\overline{u}\leq \frac{a-\chi\underline{u}+\frac{\chi\tau c}{2}(\overline{u}-\underline{u})}{b-\chi}.
\end{equation}

Next, using again inequalities \eqref{e1} and \eqref{e3}, we have that
$$
u_t\geq u_{xx}+(c-\chi v_{x})u_x+(a-\chi(\overline{u}+\varepsilon)-\frac{\chi\tau c}{2}(\overline{u}-\underline{u}+2\varepsilon)-(b-\chi)u ), \quad t\geq T_{\varepsilon}.
$$
Thus, comparison principle for parabolic equations implies that
\begin{equation}\label{e7}
u(x,t;u_0)\geq \underline{U}^{\varepsilon}(t),\quad \forall\, x\in\R,\ \forall\, t\geq T_{\varepsilon},
\end{equation}
where $\underline{U}^{\varepsilon}(t)$ is the solution of the ODE
$$
\begin{cases}
\underline{U}_{t}^\varepsilon=\underline{U}^\varepsilon(a-\chi(\overline{u}+\varepsilon)-\frac{\chi\tau c}{2}(\overline{u}-\underline{u}+2\varepsilon)-(b-\chi)\underline{U}^\varepsilon)\quad t> T_{\varepsilon}\cr
\underline{U}^\varepsilon(T_{\varepsilon})=\inf_{x\in\R}u(x,T_{\varepsilon},u_0).
\end{cases}
$$
From \eqref{e00} we know that  $\inf_{x\in\R}u(x,T_{\varepsilon},u_0)>0$.  Thus, using the fact $b-\chi>0$, we have that $\underline{U}^\varepsilon$ is defined for all time and satisfies $ \lim_{t\to\infty}\underline{U}^\varepsilon= \frac{(a-\chi(\overline{u}+\varepsilon)-\frac{\chi\tau c}{2}(\overline{u}-\underline{u}+2\varepsilon))_{+}}{b-\chi}$. Thus, it follows from \eqref{e7} that
$$
\underline{u}\geq \frac{(a-\chi(\overline{u}+\varepsilon)-\frac{\chi\tau c}{2}(\overline{u}-\underline{u}+2\varepsilon))_{+}}{b-\chi},\quad \forall\, \varepsilon>0.
$$
Letting $\varepsilon$ tends to 0 in the last inequality, we obtain that
\begin{equation}\label{e8}
\underline{u}\geq \frac{(a-\chi\overline{u}-\frac{\chi\tau c}{2}(\overline{u}-\underline{u}))_{+}}{b-\chi}.
\end{equation}
It follows from inequality \eqref{e6} and \eqref{e8} that
$$
(b-\chi)(\overline{u}-\underline{u})\leq  \chi(1+\tau c)(\overline{u}-\underline{u}).
$$
Which is equivalent to
$$
(b-\chi- \chi(1+\tau c))(\overline{u}-\underline{u})\leq 0.
$$
Since $b-\chi- \chi(1+\tau c)>0 $, then $\overline{u}=\underline{u}$. Thus it follows from \eqref{e6}  and \eqref{e8} that $\overline{u}\leq \frac{a}{b}$ and $\overline{u}\geq \frac{a}{b}$ respectively. That is $\overline{u}=\underline{u}=\frac{a}{b}$.
\end{proof}

\section{Super- and sub-solutions}

In this section, we will construct super- and sub-solutions of some related equations of \eqref{reduced-eq2}, which will be used to prove the existence of traveling wave solutions of \eqref{reduced-eq2} in next section. Throughout this section we suppose that $a>0$  and $b>0$ are given positive real numbers.

 Note that, for given $c$, to show the existence of a traveling wave solution of \eqref{main-eq1} connecting $(\frac{a}{b},\frac{a}{b})$ and
 $(0,0)$  with speed $c$ is then equivalent to
show the existence of a stationary solution of \eqref{reduced-eq2} connecting $(\frac{a}{b},\frac{a}{b})$ and $(0,0)$.

For every $\tau>0$ and $0<\mu \leq\min\{\sqrt{a}, \sqrt{\frac{1+\tau a}{(1-\tau)_{+}}}\}$, define
$$c_{\mu}=\mu+\frac{a}{\mu}\quad {\rm and}\quad \varphi_{\tau,\mu}(x)=e^{-\mu x}\quad \forall\,\, x\in\R.$$
Note that for every fixed  $\tau>0$ and $0<\mu  < \min\{\sqrt{a}, \sqrt{\frac{1+\tau a}{(1-\tau)_{+}}}\}$,  $1+\tau \mu c_\mu-\mu^2>0$ and the function $\varphi_{\tau,\mu}$ is decreasing, infinitely many differentiable, and satisfies
 \begin{equation}\label{Eq1 of varphi}
\varphi_{\tau,\mu}''(x)+c_{\mu}\varphi_{\tau,\mu}'(x)+a\varphi_{\tau,\mu}(x)=0 \quad\forall\ x\in\R
\end{equation}
and
\begin{equation}\label{Eq2 of varphi}
\frac{1}{1+\tau \mu c_\mu -\mu^2} \varphi_{\tau,\mu}''(x)+\frac{\tau  c_\mu}{1+\tau \mu c_\mu -\mu^2}\varphi' _{\tau,\mu}(x)- \frac{1}{1+\tau \mu c_\mu -\mu^2}\varphi_{\tau,\mu}(x)=-\varphi_{\tau,\mu}(x)\quad \forall\,\, x\in\R.
\end{equation}

For every $C_0>0,\ \tau>0$, and $0<\mu<\min\{\sqrt{a}, \sqrt{\frac{1+\tau a}{(1-\tau)_{+}}}\}$,    define
\begin{equation}
U_{\tau,\mu,C_0}^{+}(x)=\min\{C_0, \varphi_{\tau,\mu}(x)\}=\begin{cases}
C_0 \ \quad \text{if }\ x\leq \frac{-\ln(C_0)}{\mu}\\
e^{-\mu x} \quad \ \text{if}\ x\geq \frac{-\ln(C_0)}{\mu}.
\end{cases}
\end{equation}
and
\begin{equation}V_{\tau,\mu,C_0}^{+}(x)=\min\{C_0, \,\  \frac{1}{1+\tau \mu c_\mu-\mu^2}\varphi_{\tau,\mu}(x)\}.
\end{equation}
Since $\varphi_{\tau,\mu}$ is decreasing, then the functions $U^{+}_{\tau,\mu, C_0}$ and $V_{\tau,\mu,C_0}^{+}$ are both non-increasing. Furthermore, the functions $U^{+}_{\tau,\mu,C_0}$ and $V_{\tau,\mu,C_0}^{+}$ belong to $C^{b,\delta}_{\rm unif}(\R)$ for every $0<\delta< 1$, $\tau>0$,  $0< \mu< \min\{\sqrt{a}, \sqrt{\frac{1+\tau a}{(1-\tau)_{+}}}\}$, and $C_0>0$.

Let $C_0>0,\ \tau>0$, and  $0< \mu<\min\{\sqrt{a}, \sqrt{\frac{1+\tau a}{(1-\tau)_{+}}}\}$ be fixed. Next, let $\mu<\tilde{\mu}<\min\{2\mu,\sqrt{a}, \sqrt{\frac{1+\tau a}{(1-\tau)_{+}}}\}$ and $d>\max\{1, C_0^{\frac{\mu-\tilde{\mu}}{\mu}}\}$. The function $\varphi_{\tau,\mu}-d\varphi_{\tau,\tilde{\mu}}$ achieved its maximum value at $\bar{a}_{\mu,\tilde{\mu},d}:=\frac{\ln(d\tilde{\mu})-\ln(\mu)}{\tilde{\mu}-\mu}$ and takes the value zero at $\underline{a}_{\mu,\tilde{\mu},d}:= \frac{\ln(d)}{\tilde{\mu}-\mu}$.
Define
\begin{equation}
U_{\tau,\mu,C_0}^{-}(x):= \max\{ 0, \varphi_{\tau,\mu}(x)-d\varphi_{\tau,\tilde{\mu}}(x)\}=\begin{cases}
0\qquad \qquad \qquad \quad \text{if}\ \ x\leq \underline{a}_{\mu,\tilde{\mu},d}\\
\varphi_{\tau,\mu}(x)-d\varphi_{\tau,\tilde{\mu}}(x)\quad \text{if}\ x\geq \underline{a}_{\mu,\tilde{\mu},d}.
\end{cases}
\end{equation}
From the choice of $d$, it follows that  $0\leq U_{\tau,\mu,C_0}^{-}\leq U^{+}_{\tau,\mu,C_0}\leq C_0$ and $U_{\tau,\mu,C_0}^{-}\in C^{b,\delta}_{\rm unif}(\R)$ for every $0<\delta< 1$. Finally, let us consider the set $\mathcal{E}_{\tau,\mu}(C_0)$ defined by
\begin{equation}\label{definition-E-mu}
\mathcal{E}_{\tau,\mu}(C_0)=\{u\in C^{b}_{\rm unif}(\R) \,|\, U_{\tau,\mu,C_0}^{-}\leq u\leq U_{\tau,\mu,C_0}^{+}\}.
\end{equation}
It should be noted that $U_{\tau,\mu,C_0}^{-}$ and $\mathcal{E}_{\tau,\mu}(C_0)$ all depend on $\tilde{\mu}$ and $d$. Later on, we shall provide more information on how to choose $d$ and $\tilde{\mu}$ whenever $\tau$, $\mu$  and $C_0$ are given.

For every $u\in C_{\rm unif}^b(\R)$, consider
\begin{equation}\label{ODE2}
U_{t}=U_{xx}+(c_{\mu}-\chi V'(x;u))U_{x}+(a-\chi (V(x;u)-\tau c_{\mu}V'(x;u))-(b-\chi)U)U, \quad x\in \R,
\end{equation}
where
\begin{equation}\label{Inverse of u}
V(x;u)=\int_{0}^{\infty}\int_{\R}\frac{e^{-s}}{\sqrt{4\pi s}}e^{-\frac{|x-z|^{2}}{4s}}u(z+\tau c_{\mu}s)dzds.
\end{equation}
It is well known that the function $V(x;u)$
is the solution of the second equation of \eqref{reduced-eq2} in $C^{b}_{\rm unif}(\R)$ with given $u\in C_{\rm unif}^b(\R)$.

For  given open intervals $D\subset \R$ and $I\subset \R$, a function $U(\cdot,\cdot)\in C^{2,1}(D\times I,\R)$ is called a {\it super-solution} or {\it sub-solution} of \eqref{ODE2} on $D\times I$ if
$$U_{t}\ge U_{xx}+(c_{\mu}-\chi V'(x;u))U_{x}+(a-\chi (V(x;u)-\tau c_\mu V'(x;u))-(b-\chi)U)U \quad {\rm for}\,\, x\in D,\,\,\, t\in I
$$
or
$$
U_{t}\le U_{xx}+(c_{\mu}-\chi V'(x;u))U_{x}+(a-\chi( V(x;u)-\tau c_\mu V'(x;u))-(b-\chi)U)U \quad {\rm for}\,\, x\in D,\,\,\, t\in I.
$$
\begin{tm}
\label{super-sub-solu-thm}
Suppose that $\tau>0$, $0<\mu<\min\{\sqrt{a}, \sqrt{\frac{1+\tau a}{(1-\tau)_{+}}}\}$ and  $0<\chi<1$ satisfy
\begin{equation}\label{Eq01_Th1}
 1+ \tau c_{\mu}<\frac{b-\chi}{\chi}\quad \text{and}\quad  \Big(\frac{\mu +\tau c_\mu}{\sqrt {1+\tau\mu c_\mu -\mu^2}}+\frac{\mu( \mu +\tau c_\mu)}{1+\tau\mu c_\mu -\mu^2}\Big)\leq \frac{b- \chi}{\chi}.
\end{equation}
 Then  the following hold :

\begin{itemize}

\item[(1)]There is a positive real number $\tilde{C}_{0}>0$, $\tilde{C}_0=\tilde{C}_{0}(\tau,\mu,\chi)$, such for every $C_{0}\geq \tilde{C}_{0}$, and for every $u\in \mathcal{E}_{\tau,\mu}(C_0)$, we have that  $U(x,t)=C_0$ is  supper-solutions of \eqref{ODE2} on $\R\times\R$.

\item[(2)] For every $C_0>0$ and for every $u\in \mathcal{E}_{\tau,\mu}(C_0)$, $U(x,t)=\varphi_{\tau,\mu}(x)$ is a supper-solutions of \eqref{ODE2} on $\R\times\R$.

\item[(3)] For every $C_0>0$, there is $d_0>\max\{1,C_{0}^{\frac{\mu-\tilde{\mu}}{\mu}}\}$, $d_0=d_{0}(\tau,\mu,\chi)$, such  that for every $u\in \mathcal{E}_{\tau,\mu}(C_0)$, we have that  $U(x,t)=U_{\tau,\mu,C_0}^-(x)$ is a sub-solution of \eqref{ODE2} on
$(\underline{a}_{\mu,\tilde{\mu},d},\infty)\times \R$ for all $d\ge d_0$ and $\mu< \tilde{\mu}<\min\{\sqrt{a},\ \sqrt{\frac{1+\tau a}{(1-\tau)_{+}}},2\mu,\mu+\frac{1}{\mu+\sqrt{1+\tau\mu c_{\mu}-\mu^2}}\}$.

\item[(4)] Let $\tilde{C}_{0}$ be given by (1), then for every $u\in \mathcal{E}_{\tau,\mu}(\tilde C_0)$, $U(x,t)=U_{\tau,\mu,\tilde C_0}^-(x_\delta)$ is a sub-solution of \eqref{ODE2} on $\R\times \R$ for $0<\delta\ll 1$,
where $x_\delta=\underline{a}_{\mu,\tilde{\mu},d}+\delta$.
\end{itemize}
\end{tm}

To prove Theorem \ref{super-sub-solu-thm}, we first establish some estimates on
$V(\cdot;u)$ and $V^{'}(\cdot;u)$.

It follows from \eqref{Inverse of u}, that
\begin{equation}\label{Estimates on Inverse of u}
\max\{\|V(\cdot;u)\|_{\infty}, \ \|V'(\cdot;u)\|_{\infty} \}\leq \|u\|_{\infty}\quad \forall\ u\in C^{b}_{\rm unif}(\R).
\end{equation}
Furthermore, let
$$
C_{\rm unif}^{2,b}(\R)=\{u\in C_{\rm unif}^b(\R)\,|\, u^{'}(\cdot),\, u^{''}(\cdot)\in C_{\rm unif}^b(\R)\}.
$$
For every $u\in C_{\rm unif}^{b}(\R)$, $u\geq 0,$ we have that  $V(\cdot;u)\in C^{2,b}_{\rm unif}(\R)$ with $V(\cdot;u)\geq 0$ and
$$\| V''(\cdot;u)\|_{\infty}=\|V(\cdot;u)-\tau c_\mu V'(\cdot;u)-u\|_{\infty}\leq \max\Big\{\|V(\cdot;u)\|_{\infty},\|\tau c_\mu V'(\cdot;u)+u\|_{\infty}\Big\}.
$$
Combining this with inequality \eqref{Estimates on Inverse of u}, we obtain that
\begin{equation}\label{Estimates on Inverse of V}
\max\{\|V(\cdot;u)\|_{\infty}, \ \|V'(\cdot;u)\|_{\infty}, \|V''(\cdot;u)\|_{\infty}\}\leq (\tau c_\mu+1)\|u\|_{\infty}\quad \forall\ u\in \mathcal{E}_{\tau,\mu}(C_0).
\end{equation}

The next Lemma provide a pointwise estimate  for $|V(\cdot;u)|$ whenever $u\in \mathcal{E}_{\tau, \mu}(C_0)$.

\begin{lem}\label{Mainlem2}
For every $C_0>0$, $\tau>0$, $0<\mu<\min\{\sqrt{a},\ \sqrt{\frac{1+\tau a}{(1-\tau)_{+}}}\}$ and $u\in \mathcal{E}_{\tau,\mu}(C_0)$, let $V(\cdot;u)$ be defined as in \eqref{Inverse of u}, then
\begin{equation}\label{Eq_MainLem2}
0\leq  V(\cdot;u)\leq V^{+}_{\tau,\mu,C_0}(\cdot).
\end{equation}
\end{lem}
\begin{proof} For every  $u\in \mathcal{E}_{\tau,\mu}(C_0)$, since $ 0 \leq U^{-}_{\tau,\mu,C_0}\leq u\leq U^{+}_{\tau,\mu,C_0}$ then
$$ 0\leq V(\cdot;U^{-}_{\tau,\mu,C_0})\leq V(\cdot;u)\leq V(\cdot;U^{+}_{\tau,\mu,C_0}).$$ Hence it is enough to prove that $V(\cdot;U_{\tau,\mu,C_0}^{+})\leq V^{+}_{\tau,\mu,C_0}(\cdot)$.  For every $x\in \R$, we have that
\begin{eqnarray}\label{Eq011}
\int_{0}^{\infty}\Big(\int_{\R}\frac{e^{-s}e^{-\frac{|x-z|^2}{4s}}\varphi_{\tau,\mu}( z+\tau c_{\mu}s)}{\sqrt{4\pi s}}dz\Big)ds & = &\frac{1}{\sqrt{\pi}}\int_{0}^{\infty}e^{-s}\Big(\int_{\R}e^{- z^2}e^{-\mu (x-2\sqrt{s}z +\tau c_{\mu}s)}dz\Big)ds\nonumber\\
&=& \frac{e^{-\mu x}}{\sqrt{\pi}}\int_{0}^{\infty}e^{-(1+\tau \mu c_\mu)s}\Big(\int_{\R}e^{- |z-\mu\sqrt{s}|^{2}}e^{\mu^{2} s}dz\Big)ds\nonumber\\
&=& \frac{e^{-\mu x}}{\sqrt{\pi}}\int_{0}^{\infty} e^{-(1+\mu \tau c_\mu -\mu^2)s}\Big(\underbrace{\int_{\R}e^{- |z-\mu\sqrt{s}|^{2}}dz}_{=\sqrt{\pi}}\Big)ds\nonumber\\
&=& \frac{\varphi_{\tau,\mu}(x)}{1+\tau \mu c_\mu -\mu^2}.
\end{eqnarray}
Thus, we have
\begin{eqnarray*}
V(x;U^{+}_{\tau,\mu,C_0})&= & \int_{0}^{\infty}\Big(\int_{\R}\frac{e^{-s}e^{-\frac{|x-z|^2}{4s}}U^{+}_{\tau,\mu,C_0}( z +\tau  c_\mu s)}{\sqrt{4\pi s}}dz\Big)ds\nonumber\\
&\leq & \min\Big\{ C_0\underbrace{\int_{0}^{\infty}\int_{\R}\frac{e^{-s}e^{-\frac{|x-z|^2}{4s}}}{\sqrt{4\pi s}}dzds}_{=1}\ ,\ \int_{0}^{\infty}\Big(\int_{\R}\frac{e^{-s}e^{-\frac{|x-z|^2}{4s}}}{\sqrt{4\pi s}} \varphi_{\tau,\mu}(z +\tau  c_\mu s)dz\Big)ds\Big\}\nonumber\\
&=& V^{+}_{\tau,\mu,C_0}(x).
\end{eqnarray*}
\end{proof}

Next, we present a pointwise  estimate  for $|V'(\cdot;u)|$ whenever $u\in \mathcal{E}_{\tau,\mu}(C_0).$

\begin{lem}\label{Mainlem3} Let $\tau>0,\ C_0>0$
 and $0<\mu<\min\{\sqrt{a}, \sqrt{\frac{1+\tau a}{(1-\tau)_{+}}}\}$ be fixed. Let $u\in C^{b}_{\rm unif}(\R)$ and $V(\cdot;u)\in C^{2,b}_{\rm unif}(\R)$ be the corresponding function satisfying the second equation of \eqref{reduced-eq2}. Then
\begin{equation}\label{Eq_Mainlem01}
|V'(x;u)|\leq   \Big(\frac{1}{\sqrt {1+\tau\mu c_\mu -\mu^2}}+\frac{\mu}{1+\tau\mu c_\mu -\mu^2}\Big) \varphi_{\mu}(x)
\end{equation}
for every $x \in\R$ and every $u\in\mathcal{E}_{\tau,\mu}(C_0)$.
\end{lem}

\begin{proof} Let $u\in\mathcal{E}_{\tau,\mu}(C_0)$ and fix any $x\in \R$.
\begin{align}\label{Eq_Mainlem0002}
|V'(x;u)|
&\leq \frac{1}{\sqrt{\pi}}\int_{0}^{\infty}\int_{\R}\frac{|z|e^{-s}}{\sqrt{ s }}e^{-z^{2}}u(x+2\sqrt{s}z+\tau c_{\mu}s)dzds\nonumber\\
& \leq\frac{1}{\sqrt{\pi}}\int_{0}^{\infty}\int_{0}^\infty\frac{|z|e^{-s}}{\sqrt{ s }}e^{-z^{2}}\varphi_{\tau,\mu}(x+2\sqrt{s}z+\tau c_{\mu}s)dzds\nonumber\\
& =\frac{\varphi_{\tau,\mu}(x)}{\sqrt{\pi}}\int_{0}^{\infty}\int_{\R}\frac{|z|e^{-(1+\mu\tau c_{\mu}-\mu^2)s}}{\sqrt{s}}e^{-(z+\mu\sqrt{s})^2}dzds\nonumber\\
& \leq \frac{\varphi_{\tau,\mu}(x)}{\sqrt{\pi}}\int_{0}^{\infty}\left(\frac{e^{-(1+\mu\tau c_{\mu}-\mu^2)s}}{\sqrt{s}}\right)\left[\int_{\R}(|z|+\mu\sqrt{s})e^{-z^2}dz\right]ds\nonumber\\
&=  \frac{\varphi_{\tau,\mu}(x)}{\sqrt{\pi}}\int_{0}^{\infty}\frac{(1+\mu\sqrt{\pi s})e^{-(1+\tau \mu c_{\mu}-\mu^2)s}}{\sqrt{s}}ds \nonumber\\
&= \Big(\frac{1}{\sqrt {1+\tau\mu c_\mu -\mu^2}}+\frac{\mu}{1+\tau\mu c_\mu -\mu^2}\Big)\varphi_{\tau,\mu}(x).
\end{align}
The Lemma is thus proved.
\end{proof}

\begin{proof}[Proof of Theorem \ref{super-sub-solu-thm}]
 For every $U\in C^{2,1}(\R\times\R_{+})$, let \begin{equation}\label{mathcal L}
\mathcal{L}U=U_{xx}+(c_{\mu}-\chi V'(\cdot;u))U_{x}+(a-\chi( V(\cdot;u)-\tau c_{\mu}V'(\cdot,u))-(b-\chi)U)U .
\end{equation}
(1)  First, using inequality \eqref{Estimate-v-partial-v-eq0001}, we have that
\begin{eqnarray}\label{N001}
\mathcal{L}( C_0)&=& (a-\chi (V(\cdot;u)-\tau c_{\mu}V'(\cdot,u))-(b-\chi)C_0)C_{0}\nonumber\\
              & = & -\chi C_0 V(\cdot,u) +(a-(b-\chi(1+ \frac{\tau c_\mu}{2}))C_0)C_0.
\end{eqnarray}
Observe that $ \frac{\tau c_\mu}{2}<\frac{b-\chi}{\chi}$ is equivalent to $b>\chi(1+ \frac{\tau c_\mu}{2})>0$. Thus taking $\tilde{C}_{0}:=\frac{a}{b-\chi(1+ \frac{\tau c_\mu}{2})}$, it follows from inequality \eqref{N001} that for every $C_{0}\geq \tilde{C}_0$, we have that
$\mathcal{L}(C_{0})\leq 0.$ Hence, for every $C_{0}\geq \tilde{C}_0$, we have that $U(x,t)=C_0$ is a super-solution of \eqref{ODE2} on $\R\times\R$.

(2) It follows from Lemma \ref{Mainlem3} and \eqref{Eq01_Th1} that
\begin{eqnarray*}
 \mathcal{L}(\varphi_{\tau,\mu})& =& \varphi''_{\tau,\mu}(x)+(c_{\mu}-\chi V'(\cdot;u))\varphi_{\tau,\mu}'(x)+(a-\chi (V(\cdot;u)-\tau c_\mu V'(\cdot,u))-(b-\chi)\varphi_{\tau,\mu})\varphi_{\tau,\mu}\nonumber\\
 & =& \underbrace{(\varphi''_{\tau,\mu}+c_{\mu}\varphi'_{\tau,\mu}+a\varphi_{\tau,\mu})}_{=0} +(\mu\chi V'(\cdot;u)-\chi( V(\cdot;u)-\tau c_{\mu} V'(\cdot,u))-(b-\chi)\varphi_{\tau,\mu} )\varphi_{\tau,\mu}\nonumber\\
   & \leq & \chi\Big[ \Big(\frac{\mu +\tau c_\mu}{\sqrt {1+\tau\mu c_\mu -\mu^2}}+\frac{\mu( \mu +\tau c_\mu)}{1+\tau\mu c_\mu -\mu^2}\Big)-\frac{b-\chi}{\chi}\Big]\varphi_{\mu}^{2} \leq  0.
\end{eqnarray*}
Hence $U(x,t)=\varphi_{\tau,\mu}(x)$ is also a super-solution of \eqref{ODE2} on $\R\times\R$.

(3)  Let $C_0>0$ and $O=(\underline{a}_{\mu,\tilde{\mu},d},\infty)$. Then for $x\in O$, $U_{\tau,\mu,C_0}^-(x)>0$.
 For $x\in O$, it follows from Lemma \ref{Mainlem3} and \eqref{Eq01_Th1} that
\begin{align*}
&\mathcal{L}(U_{\tau,\mu,C_0}^{-}
)\nonumber\\
& = \mu^2\varphi_{\tau,\mu}-\tilde{\mu}^2d\varphi_{\tau,\tilde{\mu}} +(c_{\mu}-\chi V'(\cdot;u))(-\mu\varphi_{\tau,\mu}+d\tilde{\mu}\varphi_{\tau,\tilde{\mu}})\nonumber\\
&\,\,\,\, +(a-\chi( V(\cdot;u)-\tau c_{\mu}V'(\cdot,u))-(b-\chi) U_{\tau,\mu,C_0}^{-})U_{\tau,\mu,C_0}^{-} \nonumber\\
& =\underbrace{(\mu^{2}-\mu c_{\mu}+a)}_{=0}\varphi_{\tau,\mu} +d\underbrace{(\tilde{\mu}c_{\mu}-\tilde{\mu}^{2}-a)}_{=A_{0}}\varphi_{\tau,\tilde{\mu}} -\chi V'(\cdot;u)(-(\tau c_{\mu}+\mu)\varphi_{\tau,\mu}\nonumber\\
&\,\, \,\,  + d(\tau c_\mu +\tilde{\mu})\varphi_{\tau,\tilde{\mu}})   -(\chi V +(b-\chi)U_{\tau,\mu,C_0}^{-})U_{\tau,\mu,C_0}^{-}\nonumber\\
& \geq  dA_{0}\varphi_{\tau,\tilde{\mu}} -\chi|V'(\cdot;u)|((\tau c_\mu+\mu)\varphi_{\tau,\mu}+d(\tau c_\mu+\tilde{\mu})\varphi_{\tau,\tilde{\mu}}) -\chi V^{+}_{ \tau,\mu,C_0}U_{ \tau,\mu, C_0}^{-}-(b-\chi)[U^{-}_{ \tau,\mu,C_0}]^2\nonumber\\
& \geq  dA_{0}\varphi_{\tau,\tilde{\mu}}-\chi\Big(\frac{1}{\sqrt {1+\tau\mu c_\mu -\mu^2}}+\frac{\mu}{1+\tau\mu c_\mu -\mu^2}\Big)((\tau c_{\mu}+\mu)\varphi_{\tau,\mu}+d(\tau c_{\mu}+\tilde{\mu})\varphi_{\tau,\tilde{\mu}})\varphi_{\tau,\mu} \nonumber\\
&\,\, \,\, -\chi V^{+}_{ \tau,\mu,C_0}U_{\tau,\mu,C_0}^{-}-(b-\chi)[U^{-}_{\tau,\mu,C_0}]^2\nonumber\\
& \geq  dA_{0}\varphi_{\tau,\tilde{\mu}}-\chi \Big(\frac{1}{\sqrt {1+\tau\mu c_\mu -\mu^2}}+\frac{\mu}{1+\tau\mu c_\mu -\mu^2}\Big)((\tau c_{\mu}+\mu)\varphi_{\tau,\mu}+d(\tau c_\mu+\tilde{\mu})\varphi_{\tau,\tilde{\mu}})\varphi_{\tau,\mu}\nonumber\\
& \,\,\,\,   -\frac{\chi}{1+\tau\mu c_{\mu}-\mu^2}\varphi_{\tau,\mu}U^{-}_{\tau,\mu,C_0}-(b-\chi)[U^{-}_{\tau,\mu,C_0}]^2 \nonumber\\
& =  dA_{0}\varphi_{\tau,\tilde{\mu}} -\underbrace{\Big(\chi\Big(\frac{(\tau c_{\mu}+\mu)}{\sqrt {1+\tau\mu c_\mu -\mu^2}}+\frac{\mu(\tau c_{\mu}+\mu)}{1+\tau\mu c_\mu -\mu^2}\Big)+\frac{\chi}{1+\tau\mu c_{\mu}-\mu^2}+(b-\chi)\Big)}_{=A_{1}}\varphi_{\tau,\mu}^2 \nonumber\\
& \,\, \,\, +d\Big(2(b-\chi)+\frac{\chi}{1+\tau\mu c_\mu-\mu^2}-\chi\Big( \frac{(\tau c_{\mu}+\tilde{\mu})}{\sqrt{1+\tau\mu c_{\mu} -\mu^2}} +\frac{\mu(\tau c_{\mu}+\tilde{\mu})}{1+\tau\mu c_{\mu}-\mu^2}\Big)\Big)\varphi_{\tau,\mu}\varphi_{\tau,\tilde{\mu}}\nonumber\\
&\,\, \,\,  -d^2(b-\chi)\varphi_{\tau,\tilde{\mu}}^2.
\end{align*}
Note that  $U_{\tau,\mu,C_0}^{-}(x)>0$ is equivalent to $\varphi_{\tau,\mu}(x)>d\varphi_{\tau,\tilde{\mu}}(x)$, which is again equivalent to
$$
d(b-\chi)\varphi_{\tau,\mu}(x)\varphi_{\tau,\tilde{\mu}}(x)>d^{2}(b-\chi)\varphi^2_{\tau,\tilde{\mu}}(x).
$$
Since $A_{1}>0$, thus for $x\in O$, we have
\begin{eqnarray*}
\mathcal{L}U_{\mu}^{-}(x) & \geq &  dA_{0}\varphi_{\tau,\tilde{\mu}}(x) -A_{1}\varphi_{\tau,\mu}^2(x)\nonumber\\
& & +d\underbrace{\Big(b-\chi+\frac{\chi}{1+\tau\mu c_\mu-\mu^2}- \chi\Big( \frac{(\tau c_{\mu}+\tilde{\mu})}{\sqrt{1+\tau\mu c_{\mu}-\mu^2}} +\frac{\mu(\tau c_{\mu}+\tilde{\mu})}{1+\tau\mu c_{\mu}-\mu^2}\Big)\Big)}_{A_{2}}\varphi_{\tau,\mu}(x)\varphi_{\tau,\tilde{\mu}}(x)\nonumber\\
& =& A_{1}\Big[\frac{dA_{0}}{A_{1}}e^{(2\mu-\tilde{\mu})x}-1\Big]\varphi_{\tau,\mu}^{2}(x) +dA_{2}\varphi_{\tau,\mu}(x)\varphi_{\tau,\tilde{\mu}}(x).
\end{eqnarray*}
Note also that, by \eqref{Eq01_Th1},
\begin{eqnarray}\label{Eq1 of Th2}
A_{2}&=&\chi\Big(\frac{b-\chi}{\chi}-\Big( \frac{(\tau c_{\mu}+\mu)}{\sqrt{1+\tau\mu c_{\mu}-\mu^2}} +\frac{\mu(\tau c_{\mu}+\mu)}{1+\tau\mu c_{\mu}-\mu^2}\Big)\Big)+ \chi\Big(\frac{1-(\tilde{\mu}-\mu)(\mu +\sqrt{1+\tau\mu c_{\mu}-\mu^2})}{1+\tau\mu c_{\mu}-\mu^2}\Big)\nonumber\\
&\geq &  \chi\Big(\frac{ 1-(\tilde{\mu}-\mu)(\mu +\sqrt{1+\tau\mu c_{\mu}-\mu^2})}{1+\tau\mu c_{\mu}-\mu^2}\Big)\nonumber\\
&\geq & 0,
\end{eqnarray}
whenever $\tilde{\mu}\leq\mu+\frac{1}{\mu+\sqrt{1+\tau\mu c_{\mu}-\mu^2}}$.
Observe that
$$
A_{0}=\frac{(\tilde{\mu}-\mu)(a-\mu\tilde{\mu})}{\mu}>0,\quad \forall\ 0<\mu<\tilde{\mu}<\sqrt{a}.
$$
 Furthermore, we have that $U_{\tau,\mu,C_0}^{-}(x)>0$ implies that $x>0$ for $d\geq \max\{1,C_0^{\frac{\mu-\tilde{\mu}}{\mu}}\}$. Thus, for every $ d\geq d_{0}:= \max\{1, \frac{A_{1}}{A_{0}}, C_0^{\frac{\mu-\tilde{\mu}}{\mu}}\}$, we have that
\begin{equation}\label{E1}
\mathcal{L}U_{\tau,\mu,C_0}^{-}(x) > 0
\end{equation}
whenever $x\in O$ and $\mu<\tilde{\mu}< \min\{\sqrt{a},\ \sqrt{\frac{1+\tau a}{(1-\tau)_+}},2\mu,\mu+\frac{1}{\mu+\sqrt{1+\tau\mu c_{\mu}-\mu^2}}\}$. Hence $U(x,t)=U_{\tau,\mu,C_0}^-(x)$ is a sub-solution of \eqref{ODE2} on $(\underline{a}_{\mu,\tilde{\mu},d},\infty)\times\R$.

\smallskip

(4) Observe that $ 1+\tau c_{\mu}<\frac{b-\chi}{\chi}$ is equivalent to $ b-2\chi(1+ \frac{\tau c_\mu}{2})>0$. Thus, it follows from \eqref{Eq01_Th1} that
$$
a-\chi(1+ \frac{\tau c_\mu}{2})\tilde{C}_{0}=\frac{a(b-2\chi(1+ \frac{\tau c_\mu}{2}))}{b-\chi(1+\tau c_\mu)}>0.
$$ Hence, for $0<\delta\ll 1$, we have that
\begin{align*}
&(a-\chi (V(x;u)-\tau c_\mu V'(x,u))-(b-\chi)U_{\tau,\mu,\tilde{C}_0}^-(x_\delta))U_{\tau,\mu,\tilde{C}_0}^-(x_\delta)\\
 &\geq  (a-\chi(1+\frac{\tau c_\mu}{2})\tilde C_{0}-(b-\chi)U_{\tau,\mu,\tilde{C}_0}^-(x_\delta))U_{\tau,\mu,\tilde{C}_0}^-(x_\delta) \nonumber\\
&>  0\quad\forall\,\, x\in\R,
\end{align*}
where $x_\delta=\underline{a}_{\mu,\tilde{\mu},d}+\delta$. This implies that $U(x,t)=U_{\tau,\mu,\tilde{C}_0}^-(x_\delta)$ is
a sub-solution of \eqref{ODE2} on $\R\times\R$.
\end{proof}

\section{Traveling wave solutions}

In this section we study the existence and nonexistence  of traveling wave solutions of \eqref{reduced-eq2} connecting $(\frac{a}{b},\frac{a}{b})$ and
$(0,0)$, and prove Theorems C and D.

\subsection{Proof of Theorem C}

In this subsection, we  prove Theorem C. To this end,  we first prove  the following important result.

\begin{tm}
\label{existence-tv-thm}
Let $a>0$ and $b>0$ be given. Suppose that $\tau>0$, $0<\chi<1 $ and $0< \mu<\min\{\sqrt{a},\sqrt{\frac{1+\tau a}{(1-\tau)_+}}\}$ satisfy  \eqref{Eq01_Th1}. Then \eqref{main-eq1} has a traveling wave solution
$(u(x,t),v(x,t))=(U(x-c_\mu t),V(x-c_\mu t))$ satisfying
$$
\lim_{x\to-\infty}U(x)=\frac{a}{b} \quad \text{and}\quad
\lim_{x\to\infty}\frac{U(x)}{e^{-\mu x}}=1
$$
where $c_{\mu}=\mu+\frac{a}{\mu}$.
\end{tm}

Our key idea to prove the above theorem is to prove that, for any $\tau>0$, $0<\mu< \min\{\sqrt{a},\sqrt{\frac{1+\tau a}{(1-\tau)_+}}\}$ and $0<\chi<1$ satisfying \eqref{Eq01_Th1}, there is $u^*(\cdot)\in\mathcal{E}_{\tau,\mu}(\tilde{C}_0)$ such that $(U(\cdot),V(\cdot))=(u^*(\cdot),V(\cdot;u^*))$ is a stationary solution of \eqref{reduced-eq2} with $c=c_\mu$, where
 $V(\cdot;u^*)$ is given by \eqref{Inverse of u}, and $u^*(-\infty)=\frac{a}{b}$ and $u^*(\infty)=0$, which implies that
$(u(x,t),v(x,t))=(u^*(x-c_\mu t),V(x-c_\mu t;u^*))$ is a traveling wave solution of \eqref{main-eq1} connecting $(\frac{a}{b},\frac{a}{b})$ and $(0,0)$
with speed $c=c_\mu$.

In order to prove Theorem \ref{existence-tv-thm}, we first prove some lemmas. These Lemmas extend some of the results established in \cite{SaSh2}, so some details might be omitted in their proofs. The reader is referred to the proofs of Lemmas 3.2, 3.3, 3.5 and 3.6 in \cite{SaSh2} for more details.

 In the remaining part of this subsection we shall suppose that \eqref{Eq01_Th1}  holds and $\tilde \mu$ is fixed, where $\tilde \mu$ satisfies
   $$\mu<\tilde{\mu}<\min\{\sqrt{a},\sqrt{\frac{1+\tau a}{(1-\tau)_+}},2\mu,\mu+\frac{ 1}{\mu+\sqrt{1+\tau\mu c_{\mu}- \mu^2}}\}.
    $$
    Furthermore, we choose $C_{0}=\tilde{C}_{0}$ and $d=d_0(\tau,\mu,\chi,\tilde{C}_0)$ to be the constants  given by Theorem \ref{super-sub-solu-thm}
     and to be fixed. Fix $u\in\mathcal{E}_{\tau,\mu}(C_0)$. For given $u_0\in C_{\rm unif}^b(\R)$, let
  $U(x,t;u_0)$ be the solution of \eqref{ODE2} with
$U(x,0;u_0)=u_0(x)$. By the arguments in the proof of Theorem 1.1 and Theorem 1.5 in \cite{SaSh}, we have $U(x,t; U_{\tau,\mu,C_0}^+)$ exists for all $t>0$ and
${ U(\cdot,\cdot;{ U_{\tau,\mu,C_0}^+})}\in C([0,\infty),C^{b}_{\rm unif}(\R))\cap C^{1}((0\ ,\ \infty),C^{b}_{\rm unif}(\R))\cap C^{2,1}(\R\times(0,\ \infty))$ satisfying
\begin{equation}
U(\cdot,\cdot; U_{\tau,\mu,C_0}^+), U_{x}(\cdot,\cdot; U_{\tau,\mu,C_0}^+),U_{xx}(\cdot,t; U_{\tau,\mu,C_0}^+),U_{t}(\cdot,\cdot; U_{\tau,\mu,C_0}^+)\in  C^{\theta}((0, \infty),C_{\rm unif}^{b,\nu}(\R))
\end{equation}
for $0<\theta, \nu \ll 1$.

\begin{lem} \label{lm1}
Assume that $\tau>0$, $0<\mu<\min\{\sqrt{a},\sqrt{\frac{1+\tau a}{(1-\tau)_+}}\}$, $0<\chi<1$ satisfy
\eqref{Eq01_Th1}.
 Then for every $u\in \mathcal{E}_{\tau,\mu}(C_0)$, the following hold.
\begin{description}
\item[(i)] $0\leq U(\cdot,t; U_{\tau,\mu,C_0}^+)\leq U_{ \tau,\mu,C_0}^{+}(\cdot)$ for every $t\geq 0.$
\item[(ii)] $U(\cdot,t_{2}; U_{\tau,\mu,C_0}^+)\leq U(\cdot,t_{1}; U_{\tau,\mu,C_0}^+) $ for every $0\leq t_{1}\leq t_{2}$.
\end{description}
\end{lem}
\begin{proof}
(i)   Note that $U^{+}_{\tau,\mu,C_0}(\cdot)\leq  C_0 $. Then by
comparison principle for parabolic equations and Theorem \ref{super-sub-solu-thm}(1), we have
\begin{equation*}
U(x,t; U_{\tau,\mu,C_0}^+)\leq  C_0\quad \forall\ x\in\R,\ t\geq 0.
\end{equation*}

Similarly, note that $U_{\tau,\mu,C_0}^+(x)\le\varphi_{\tau,\mu}(x)$.  Then by  comparison principle for parabolic equations and
Theorem \ref{super-sub-solu-thm}(2)  again, we have
\begin{equation*}
U(x,t;U_{\tau,\mu,C_0}^+)\leq \varphi_{\tau,\mu}(x) \ \quad \forall\ x\in\R,\ t\geq 0.
\end{equation*}
Thus $U(\cdot,t;U_{\tau,\mu,C_0}^+)\leq U^{+}_{\tau,\mu,C_0}$. This complete of (i).

(ii)  For $0\leq t_{1}\leq t_{2}$, since
$$
U(\cdot,t_{2};U_{\tau,\mu,C_0}^+)=U(\cdot,t_{1},U(\cdot,t_{2}-t_{1};U_{\tau,\mu,C_0}^+))
$$
and by (i), $U(\cdot,t_{2}-t_{1};U_{\tau,\mu,C_0}^+)\leq U^{+}_{\tau,\mu,C_0} $, (ii) follows from comparison principle for parabolic equations.
\end{proof}

Let us define $U(x; u)$ to be
\begin{equation}
\label{U-eq}
{
U(x; u)=\lim_{t\to\infty}U(x,t; U^{+}_{ \tau,\mu,C_0})=\inf_{t>0}U(x,t; U^{+}_{ \tau,\mu,C_0}).
}
\end{equation}
 By the a priori estimates for parabolic equations, the limit in \eqref{U-eq} is uniform in $x$ in compact subsets of $\R$
and $U(\cdot;u)\in C_{\rm unif}^b(\R)$.
Next we prove that the function $u\in \mathcal{E}_{\tau,\mu}(C_0)\to U(\cdot;u)\in\mathcal{E}_{\tau,\mu}(C_0)$.

\begin{lem}\label{lm2}
Assume that $\tau>0$, $0<\mu<\min\{\sqrt{a},\sqrt{\frac{1+\tau a}{(1-\tau)_+}}\}$, $0<\chi<1$ satisfy
\eqref{Eq01_Th1}. Then,
\begin{equation}
U(x;u)\geq\begin{cases}  U^{-}_{{ \tau,\mu,C_0}}(x),\quad x\ge \underline{a}_{\mu,\tilde{\mu},d}\cr
U_{{ \tau,\mu,C_0}}^-(x_\delta),\quad x\le x_\delta=\underline{a}_{\mu,\tilde{\mu},d}+\delta
\end{cases}
\end{equation}\label{Eq2 of Th2}
for every $u\in\mathcal{E}_{\tau,\mu}(C_0)$,  $t\geq 0$, and $0<\delta\ll 1$.
\end{lem}

\begin{proof} Let $u\in \mathcal{E}_{\tau,\mu}(C_0)$ be fixed.  Let $O=(\underline{a}_{\mu,\tilde{\mu},d},\infty)$.
Note that $U_{\tau,\mu,C_0}^-(\underline{a}_{\mu,\tilde{\mu},d})=0$. By Theorem \ref{super-sub-solu-thm}(3),
$U_{\tau,\mu,C_0}^-(x)$ is a sub-solution of \eqref{ODE2} on $O\times (0,\infty)$.
Note also that $U_{\tau,\mu,C_0}^+(x)\ge U_{\tau,\mu,C_0}^-(x)$ for $x\ge \underline{a}_{\mu,\tilde{\mu},d}$ and $U(\underline{a}_{\mu,\tilde{\mu},d},t;U_{\tau,\mu,C_0}^+))>0$
for all $t\ge 0$. Then by comparison principle for parabolic equations, we have that
$$
U(x,t;U_{\tau,\mu,C_0}^+)\ge U_{\tau,\mu,C_0}^-(x)\quad \forall \,\, x\ge \underline{a}_{\mu,\tilde{\mu},d},\,\, t\ge 0.
$$

Now for any $0<\delta\ll 1$, by Theorem \ref{super-sub-solu-thm}(3), $U(x,t)=U_{\tau,\mu,C_0}^-(x_\delta)$ is a sub-solution of
\eqref{ODE2} on $\R\times R$. Note that $U_{\tau,\mu,C_0}^+(x)\ge U_\mu^-(x_\delta)$ for $x\le x_\delta$ and
$U(x_\delta,t;U_{\tau,\mu,C_0}^+)\ge U_{\tau,\mu,C_0}^-(x_\delta)$ for $t\ge 0$. Then by comparison principle for parabolic equations again,
$$
U(x,t;U_{\tau,\mu,C_0}^+)\ge U_{\tau,\mu,C_0}^-(x_\delta)\quad \forall\,\, x\le x_\delta,\, \, t>0.
$$
The lemma then follows.
\end{proof}

\begin{rk}\label{Remark-lower-bound-for -solution}
 It follows from Lemmas \ref{lm1} and \ref{lm2} that if \eqref{Eq01_Th1} holds, then
$$
U_{\tau,\mu,C_0,\delta}^{-}(\cdot)\leq U(\cdot,t;U_{\tau,\mu,C_0}^+)\leq U^{+}_{\tau,\mu,C_0}(\cdot)$$
for every $u\in\mathcal{E}_{\tau,\mu}(C_0)$, $t\geq0$ and $0\le \delta\ll 1$, where
$$
U_{\tau,\mu,C_0,\delta}^-(x)=\begin{cases}  U^{-}_{\mu}(x),\quad\ x\ge \underline{a}_{\mu,\tilde{\mu},d} + \delta,\cr
U_\mu^-(x_\delta),\quad x\le x_\delta=\underline{a}_{\mu,\tilde{\mu},d}+\delta.
\end{cases}
$$
 This implies that $$
U_{\tau,\mu,C_0,\delta}^{-}(\cdot)\leq U(\cdot;u)\leq U^{+}_{\tau,\mu,C_0}(\cdot)$$
for every $u\in\mathcal{E}_{\tau,\mu}(C_0)$. Hence  $u\in\mathcal{E}_{\tau,\mu}(C_0)\mapsto U(\cdot;u)\in \mathcal{E}_{\tau,\mu}(C_0).$
\end{rk}

\medskip

\begin{lem}\label{MainLem02}
Assume that $\tau>0$, $0<\mu<\min\{\sqrt{a},\sqrt{\frac{1+\tau a}{(1-\tau)_+}}\}$, $0<\chi<1$ satisfy
\eqref{Eq01_Th1}. Then for every $u\in \mathcal{E}_{\tau,\mu}(C_0)$ the associated function $U(\cdot;u)$ satisfied the elliptic equation,
\begin{equation}\label{Eq_MainLem02}
0=U_{xx}+(c_{\mu}-\chi V'(x;u))U_{x}+(a-\chi( V(x;u)-\tau c_\mu V'(\cdot,u))-(b-\chi)U)U\quad \forall\,\, x\in\R.
\end{equation}
\end{lem}
\begin{proof} The following arguments generalized the arguments used in the proof of Lemma 4.6 in \cite{SaSh2} to cover the case  $\tau>0$. Hence we refer to \cite{SaSh2} for the proofs of the estimates stated below.

 Let $\{t_{n}\}_{n\geq 1}$ be an increasing sequence of positive real numbers converging to $\infty$. For every $n\geq 1$, define $U_{n}(x,t)=U(x,t+t_{n};u)$ for every $x\in\R, \ t\geq 0$.
For every $n$, $U_{n}$ solves the PDE
\begin{equation*}
\begin{cases}
\partial_{t}U_{n}=\partial_{xx}U_{n}+(c_{\mu}-\chi V'(x;u))\partial_{x}U_{n}+(1-\chi (V(x;u)-\tau c_\mu V'(\cdot,u))-(1-\chi)U_{n})U_{n}\ \ x\in\R\\
U_{n}(\cdot,0)=U(\cdot,t_{n};u).
\end{cases}
\end{equation*}

Let $\{T(t)\}_{t\geq 0}$ be the analytic semigroup on $C^{b}_{\rm unif}(\R)$ generated by $\Delta-I$
 and  let $X^{\beta}={\rm Dom}((I-\Delta)^{\beta})$ be the fractional power spaces of $I-\Delta$ on $C_{\rm unif}^b(\R)$ ($\beta\in [0,1]$).

 The variation of constant formula and the fact that $V''(x;u)+\tau c_{\mu}V'(x;u)-V(x;u)=-u(x)$ yield that
\begin{eqnarray}\label{variation -of-const}
U(\cdot,t;u)
&=& \underbrace{T(t)U_{\mu}^{+}}_{I_{1}(t)}+ \underbrace{\int_{0}^{t}T(t-s)(((c_{\mu}-\chi V'(\cdot;u)U)_{x})ds}_{I_{2}(t)} +\underbrace{\int_{0}^{t}T(t-s)(1+a-\chi u)U(\cdot,s;u)ds}_{I_{3}(t)}\nonumber \\
& &-(b-\chi)\underbrace{\int_{0}^{t}T(t-s)U^{2}(\cdot,s;u)ds}_{I_{4}(t)}.
\end{eqnarray}
Let $0<\beta<\frac{1}{2}$ be fixed. There is a positive constant $C_{\beta}$, $C_{\beta}=C_{\beta}(N)$ such that
\begin{equation*}
\|I_{1}(t)\|_{X^{\beta}}\leq \frac{aC_\beta t^{-\beta}e^{-t}}{b-\chi(1+\frac{\tau c_\mu}{2})},\quad \|I_{2}(t)\|_{X^{\beta}}\leq  \frac{aC_{\beta}}{b-\chi(1+\frac{\tau c_\mu}{2})}(c_{\mu}+\frac{a\chi}{b-\chi(1+\frac{\tau c_\mu}{2})})\Gamma(\frac{1}{2}-\beta),
\end{equation*}
and
$$
\|I_{3}(t)\|_{X^{\beta}}
 \leq  \frac{aC_{\beta}((1+a)(b-\chi(1+\frac{\tau c_\mu}{2})) +a\chi)}{(b-\chi(1+\frac{\tau c_\mu}{2}))^2}\Gamma(1-\beta), \ \ \|I_{4}(t)\|_{X^{\beta}}\leq \frac{a^2 C_{\beta}}{(b-\chi(1+\frac{\tau c_\mu}{2}))^{2}}\Gamma(1-\beta).
$$
Therefore, for every $T>0$ we have that
\begin{equation}\label{Eq_Convergence01}
\sup_{t\geq T}\|U(\cdot,t;u)\|_{X^{\beta}}\leq M_{T}<\infty,
\end{equation}
where
\begin{equation}\label{Eq_Conv02}
M_{T}= \frac{aC_{\beta}}{b-\chi(1+\frac{\tau c_\mu}{2})}\Big[T^{-\beta}e^{-T}+ ( c_{\mu}+  \frac{(1+2a)(b+1)+\chi}{b-\chi(1+\frac{\tau c_{\mu}}{2})})( \Gamma(1-\beta)+\Gamma(\frac{1}{2}-\beta))\Big].
\end{equation}
Hence it follows that
\begin{equation}\label{Eqq000}
\sup_{n\geq 1, t\geq 0}\|U_{n}(\cdot,t)\|_{X^{\beta}}\leq M_{t_{1}}<\infty.
\end{equation}

Next, for every $t,h\geq 0$ and $n\geq 1$, we have that
\begin{equation}\label{Eqq00}
\|I_{1}(t+h+t_{n})-I_{1}(t+t_{n})\|_{X^{\beta}}\leq C_{\beta}h^{\beta}(t+t_{n})^{-\beta}e^{-(t+t_n)}\|U_{\mu}^{+}\|_{\infty}\leq C_{\beta}h^{\beta}t_{1}^{-\beta}e^{-t_1}\|U_{\mu}^{+}\|_{\infty},
\end{equation}
\begin{align}\label{Eqq02}
&\|I_{2}(t+t_n+h)-I_{2}(t+t_n)\|_{X^{\beta}}\nonumber\\
&\leq \frac{aC_{\beta}}{b-\chi(1+\frac{\tau c_\mu}{2})}(c_\mu+\frac{a\chi}{b-\chi(1+\frac{\tau c_\mu}{2})})\Big[h^{\beta}\Gamma(1-\beta)+\frac{h^{\frac{1}{2}-\beta}}{\frac{1}{2}-\beta} \Big] ,
\end{align}
\begin{align}\label{Eqq01}
&\|I_{3}(t+h+t_n)-I_{3}(t+t_n)\|_{X^{\beta}}\nonumber\\
&\leq \frac{aC_{\beta}((1+a)(b-\chi(1+\frac{\tau c_\mu}{2})) +a\chi)}{(b-\chi(1+\frac{\tau c_\mu}{2}))^2}\Big[h^{\beta}\Gamma(1-\beta)+\frac{h^{1-\beta}}{1-\beta} \Big] ,
\end{align}
and
\begin{eqnarray}\label{Eqq03}
\|I_{4}(t+t_n+h)-I_{4}(t+t_n)\|_{X^{\beta}}\leq\frac{a^2 C_{\beta}}{(b-\chi(1+\frac{\tau c_\mu}{2}))^2}\Big[h^{\beta}\Gamma(1-\beta)+\frac{h^{1-\beta}}{1-\beta} \Big].
\end{eqnarray}
It follows from inequalities \eqref{Eqq000}, \eqref{Eqq00}, \eqref{Eqq01}, \eqref{Eqq02} and \eqref{Eqq03}, the functions $U_{n} : [0, \infty)\to X^{\beta}$ are uniformly bounded and equicontinuous.
Since $X^{\beta}$ is continuously imbedded in $C^{b,\nu}_{\rm unif}(\R)$ for every $0\leq \nu<2\beta$ (See \cite{Dan Henry}),
therefore, the Arzela-Ascoli Theorem and Theorem 3.15 in   \cite{Friedman}, imply that there is a function $\tilde{U}(\cdot,\cdot;u)\in C^{2,1}(\R\times(0,\infty))$ and a subsequence $\{U_{n'}\}_{n\geq 1}$ of $\{U_{n}\}_{n\geq 1}$ such that $U_{n'}\to \tilde{U}$ in $C^{2,1}_{loc}(\R\times(0, \infty))$ as $n\to \infty$ and $\tilde{U}(\cdot,\cdot;u)$ solves the PDE
$$
\begin{cases}
\partial_{t}\tilde{U}=\partial_{xx}\tilde{U}+(c_{\mu}-\chi V'(\cdot;u))\partial_{x}\tilde{U}+(a-\chi( V(\cdot;u) -\tau V'(\cdot;u))-(b-\chi)\tilde{U})\tilde{U}\ \ x\in \R\ , \ t>0\\
\tilde{U}(x,0)=\lim_{n\to \infty}U(x,t_{n'};u).
\end{cases}
$$
But $U(x;u)=\lim_{t\to \infty}U(x,t;u)$ and $t_{n'}\to \infty$ as $n\to \infty$, hence $\tilde{U}(x,t;u)=U(x;u)$ for every $x\in \R,\ t\geq 0$. Hence $U(\cdot;u)$ solves \eqref{Eq_MainLem02}.
\end{proof}

\begin{lem}
\label{aux-lm} Assume that $\tau>0$, $0<\mu<\min\{\sqrt{a},\sqrt{\frac{1+\tau a}{(1-\tau)_+}}\}$, $0<\chi<1$ satisfy
\eqref{Eq01_Th1}. Then, for any given $u\in\mathcal{E}_{\tau,\mu}(C_0)$,
\eqref{Eq_MainLem02} has a unique bounded non-negative solution satisfying that
\begin{equation}
\label{aux-eq1}
\liminf_{x\to -\infty}U(x)>0\quad {\rm and}\quad \lim_{x\to\infty}\frac{U(x)}{e^{-\mu x}}=1.
\end{equation}
\end{lem}
The proof of Lemma \ref{aux-lm} follows from  \cite[Lemma 3.6]{SaSh2}.

We now prove Theorem \ref{existence-tv-thm}.

\begin{proof}[Proof of Theorem  \ref{existence-tv-thm}]
Following the proof of Theorem 3.1 in \cite{SaSh2},
 let us consider the normed linear space  $\mathcal{E}=C^{b}_{\rm unif}(\R)$ endowed with the norm
$$\|u\|_{\ast}=\sum_{n=1}^{\infty}\frac{1}{2^n}\|u\|_{L^{\infty}([-n,\ n])}. $$
For every $u\in\mathcal{E}_{\tau, \mu}(C_0)$ we have that
$$\|u\|_{\ast}\leq C_0=\frac{a}{b-\chi(1+\tau c_{\mu})}. $$
Hence $\mathcal{E}_{\tau,\mu}(C_0)$ is a bounded convex subset of $\mathcal{E}$. Furthermore, since the convergence in $\mathcal{E}$ implies the
  pointwise convergence, then $\mathcal{E}_{\tau,\mu}(C_0)$ is a closed, bounded, and convex subset of $\mathcal{E}$. Furthermore, a sequence of functions in $\mathcal{E}_{\tau,\mu}(C_0)$ converges with respect to norm $\|\cdot\|_{\ast}$ if and only if it  converges locally uniformly convergence on $\R$.

We prove that the mapping $\mathcal{E}_{\tau,\mu}(C_0)\ni u\mapsto U(\cdot;u)\in\mathcal{E}_{\tau,\mu}(C_0)$ has a fixed point. We divide the proof in two steps.

\smallskip

\noindent {\bf Step 1.} In this step, we prove that the mapping $\mathcal{E}_{\tau,\mu}(C_0)\ni u\mapsto U(\cdot;u)\in \mathcal{E}_{\tau,\mu}(C_0)$ is compact.

 Let $\{u_{n}\}_{n\geq 1}$ be a sequence of elements of $\mathcal{E}_{\tau,\mu}(C_0)$. Since $U(\cdot;u_{n})\in \mathcal{E}_{\tau,\mu}(C_0)$ for every $n\geq 1$ then $\{U(\cdot;u_{n})\}_{n\geq 1}$ is clearly uniformly bounded by $C_0$. Using inequality \eqref{Eq_Convergence01}, we have that
\begin{equation*}
\sup_{t\geq 1}\|U(\cdot,t;u_{n})\|_{X^{\beta}}\leq M_{1}
\end{equation*}
for all $n\geq 1$ where $M_{1}$ is given by \eqref{Eq_Conv02}. Therefore there is $0<\nu\ll 1$ such that
\begin{equation}\label{Proof-MainTh3- Eq1}
\sup_{t\geq 1}\|U(\cdot,t;u_{n})\|_{C^{b,\nu}_{\rm unif}(\R)}\leq \tilde{M_{1}}
\end{equation} for every $n\geq 1$ where $\tilde{M_{1}}$ is a constant depending only on $M_{1}$. Since for every $n\geq 1$ and every $x\in\R$, we have that $U(x,t;u_{n})\to U(x;u_{n})$ as $t\to \infty,$ then it follows from \eqref{Proof-MainTh3- Eq1} that
\begin{equation}\label{Prof-MainTh3- Eq2}
\|U(\cdot;u_{n})\|_{C^{b,\nu}_{\rm unif}}\leq \tilde{M_{1}}
\end{equation} for every $n\geq 1$. Which implies that the sequence $\{U(\cdot;u_{n})\}_{n\geq 1}$ is equicontinuous. The Arzela-Ascoli's Theorem implies that there is a subsequence $\{U(\cdot;u_{n'})\}_{n\geq 1}$ of the sequence $\{U(\cdot;u_{n})\}_{n\geq 1}$ and a function $U\in C(\R)$ such that $\{U(\cdot;u_{n'})\}_{n\geq 1}$ converges to $U$ locally uniformly on $\R$. Furthermore, the function $U$ satisfies inequality \eqref{Prof-MainTh3- Eq2}. Combining this with the fact  $U_{\tau,\mu,C_0}^{-}(x)\leq U(x;u_{n'})\leq U_{\tau,\mu,C_0}^{+}(x)$ for every $x\in\R$ and $n\geq 1$, by letting $n$ goes to infinity, we obtain that  $U\in \mathcal{E}_{\tau,\mu}(C_0)$.

\smallskip

\noindent{\bf Step 2.} In this step, we prove that the mapping $\mathcal{E}_{\tau,\mu}(C_0)\ni u\mapsto U(\cdot;u)\in \mathcal{E}_{\tau,\mu}(C_0)$ is continuous.
 This follows from the arguments used in the proof of Step 2, Theorem 3.1, \cite{SaSh2}

Now by Schauder's Fixed Point Theorem, there is $U\in\mathcal{E}_{\tau,\mu}(C_0)$ such that $U(\cdot;U)=U(\cdot)$. Then
$(U(x),V(x;U))$ is a stationary solution of \eqref{stationary-eq} with $c=c_\mu$. It is clear that
$$
\lim_{x\to\infty}\frac{U(x)}{e^{-\mu x}}=1.
$$

We claim that
$$
\lim_{x\to -\infty}U(x)=\frac{a}{b}.
$$
For otherwise, we may assume that there is $x_n\to -\infty$ such that $U(x_n)\to \lambda\not =\frac{a}{b}$ as $n\to\infty$. Define $U_{n}(x)=U(x+x_{n})$ for every $x\in\R$ and $n\geq 1$. By observing that $U_{n}=U(\cdot;U_{n})$ for every $n$, hence it follows from the step 1, that there is a subsequence $\{U_{n'}\}_{n\geq 1}$ of $\{U_{n}\}_{n\geq 1}$ and a function $U^*\in\mathcal{E}_{\mu}$ such that $\|U_{n'}-U^*\|_{\ast}\to 0$ as $n\to \infty$. Next, it follows from step 2 that $(U^*,V(\cdot;U^*))$ is also a stationary solution of \eqref{stationary-eq}.

\smallskip

\noindent {\bf Claim.} $\inf_{x\in\R}U^{*}(x)>0$. Indeed, let $0< \delta\ll 1$ be fixed. For  every $x\in\R$, there $N_{x}\gg 1$ such that $x+x_{n'}< x_{\delta}$ for all $n\geq N_{x}$. Hence, It follows from Remark \ref{Remark-lower-bound-for -solution} that  $$0< U_{\tau,\mu,C_0,\delta}^{-}(x_{\delta})\leq U(x+x_{n'}) \ \forall\ n\geq N_{x}.$$
Letting $n$ goes to infinity in the last inequality, we obtain that $U_{\tau,\mu,C_0,\delta}^{-}(x_{\delta})\leq U^{*}(x)$ for every $x\in\R$. The claim thus follows.

Since $0<\mu<\min\{\sqrt{a},\sqrt{\frac{1+\tau a}{(1-\tau)_{+}}}\}$  and $0<\chi<1$ satisfy \eqref{Eq01_Th1}, then we have that $b-\chi-\chi(1+\tau c_{\mu})>0$. Thus, it follows from Theorem B that $U^*(x)=V(x;U^*)=\frac{a}{b} $ for every $x\in \R$.
This implies that $\lambda=U^*(0)=\frac{a}{b}$, which is a contradiction. Hence $\lim_{x\to -\infty} U(x)=\frac{a}{b}$.
\end{proof}

 Now, we are ready to prove Theorem C.

\begin{proof}[Proof of Theorem C]
First, for each $\tau>0$, let us define
$\mu_{\tau}:=\min\{\sqrt{a}\ ,\ \sqrt{\frac{1+\tau a}{(1-\tau)_+}}\} $ and $
\chi^*_\tau : =\frac{b}{1+m_{\tau}}
$  where $m_{\tau}$ is given by
\begin{align*}
m_{\tau}:=\inf\Big\{\max\Big\{ 1+\tau c_{\mu}\ ,\ \Big(\frac{\mu +\tau c_\mu}{\sqrt{1+\tau\mu c_\mu-\mu^2}}+\frac{\mu(\mu+\tau c_\mu)}{1+\tau\mu c_\mu-\mu^2}\Big) \Big\} \,\,\,  | \, \mu\in\Big(0\ ,\ \mu_\tau\Big) \Big\},
\end{align*}
with $c_\mu=\mu+\frac{a}{\mu}$.
Observe that $ m_{\tau}=\frac{b-\chi^*_\tau}{\chi^*_\tau}$ for every $\tau>0$. Let $\tau >0$ be fixed.  Since  the function $(0\ ,\  \infty)\ni\chi\mapsto \frac{b-\chi}{\chi}$ is continuous and strictly decreasing, then for every $0<\chi<\chi^*_{\tau}$ we have that $m_\tau <\frac{b-\chi}{\chi}$.  Hence, it follows from the continuity of the function
$$
\big(0\ ,\ \mu_{\tau}\big)\ni \mu\mapsto \max\Big\{ 1+\tau c_{\mu}\ ,\ \Big(\frac{\mu +\tau c_\mu}{\sqrt{1+\tau\mu c_\mu-\mu^2}}+\frac{\mu(\mu+\tau c_\mu)}{1+\tau\mu c_\mu-\mu^2}\Big) \Big\}
$$
 that for every $0<\chi<\chi^*_{\tau}$, there is a nonempty open interval $I_{\chi}\subset (0 \ ,\ \mu_\tau)$ such that
 \begin{equation}\label{Zaux-01}
\max\Big\{ 1+\tau c_{\mu}\ ,\ \Big(\frac{\mu +\tau c_\mu}{\sqrt{1+\tau\mu c_\mu-\mu^2}}+\frac{\mu(\mu+\tau c_\mu)}{1+\tau\mu c_\mu-\mu^2}\Big) \Big\}<\frac{b-\chi}{\chi}, \quad \forall\ \mu\in I_{\chi}.
 \end{equation}

Next, for every $ 0< \chi<\chi^*_\tau$ define $ \tilde \mu^{**}(\chi;\tau)$  and $E_{\chi;\tau}$ by
\begin{align}\label{ZZZ0}
 \tilde \mu^{**}(\chi;\tau)=\inf\Big\{& \mu\in(0\ ,\ \mu_\tau)\ |\ \max\big\{ 1+\tau c_{\mu}\ ,\ \frac{\mu +\tau c_\mu}{\sqrt{1+\tau\mu c_\mu-\mu^2}}+\frac{\mu(\mu+\tau c_\mu)}{1+\tau\mu c_\mu-\mu^2} \big\}<\frac{b-\chi}{\chi}   \Big\}.
\end{align}
and
$$
E_{\chi;\tau}:=\big\{\mu\in (0\ , \mu_{\tau})\ | \ \Big(\frac{\mu +\tau c_\mu}{\sqrt{1+\tau\mu c_\mu-\mu^2}}+\frac{\mu(\mu+\tau c_\mu)}{1+\tau\mu c_\mu-\mu^2}\Big)=\frac{b-\chi}{\chi}\big\}.
$$
 We note that
\begin{equation}
\label{ZZZ00-1}
1+\tau c_\mu <\frac{b-\chi}{\chi}\quad \forall \mu \in (\tilde \mu^{**}(\chi;\tau),\mu_\tau).
\end{equation}
We also note that
\begin{equation}\label{ZZZ00}
\lim_{\mu\to 0^+} \Big(\frac{\mu +\tau c_\mu}{\sqrt{1+\tau\mu c_\mu-\mu^2}}+\frac{\mu(\mu+\tau c_\mu)}{1+\tau\mu c_\mu-\mu^2}\Big)=\infty,
\end{equation}
and that the set $E_{\chi,\tau}$ has at most finitely many elements. It follows from \eqref{Zaux-01} and \eqref{ZZZ00} that  for every  $\chi<\chi_{\tau}^{*}$, $E_{\chi,\tau}$ is nonempty. Hence, for every $\chi<\chi_{\tau}^*$, there is $n=n(\chi,\tau)\in\N$ and $0<\mu_{\chi;\tau}^{1}<\cdots< \mu_{\chi;\tau}^{n}<\mu_{\tau}$ with
$$ \frac{ {\mu_{\chi,\tau}^{i}} +\tau c_{\mu_{\chi,\tau}^{i}}}{\sqrt{1+\tau{\mu_{\chi,\tau}^{i}}c_{\mu_{\chi,\tau}^{i}}-({\mu_{\chi,\tau}^{i}})^2}}+\frac{{\mu_{\chi,\tau}^{i}}({\mu_{\chi,\tau}^{i}}+\tau c_\mu)}{1+\tau\mu c_{\mu_{\chi,\tau}^{i}}-({\mu_{\chi,\tau}^{i}})^2}=\frac{b-\chi}{\chi}, \quad \forall\ i=1,\cdots,n . $$  Let us set $\mu_{\chi;\tau}^{0}=0$ and $\mu_{\chi,\tau}^{1+n(\chi,\tau)}=\mu_\tau$.  It follows from \eqref{ZZZ00} that for every $\chi<\chi_{\tau}^{*}$, we have  $\mu_{\chi,\tau}^{1}\leq  \tilde \mu^{**}(\chi;\tau)$. Note that the function $(0\ ,\ \mu_\tau)\ni\mu \mapsto \frac{\mu +\tau c_\mu}{\sqrt{1+\tau\mu c_\mu-\mu^2}}+\frac{\mu(\mu+\tau c_\mu)}{1+\tau\mu c_\mu-\mu^2} $ has constant sign on each of the open intervals  $(\mu_{\chi,\tau}^{i}\ ,\ \mu_{\chi,\tau}^{i+1})$, $i=1,\cdots,n(\chi,\tau)$.
Since $n(\chi;\tau)$ is finite, and the intervals $(0\ ,\ \mu_{\chi,\tau}^{1}), [{\mu_{\chi,\tau}^{1}}\ ,\ \mu_{\chi,\tau}^{2})$, $[\mu_{\chi,\tau}^{2}\ ,\ \mu_{\chi,\tau}^{3})$,$\cdots$, $[\mu_{\chi,\tau}^{n(\chi;\tau)}\ ,\ {\mu_{\chi,\tau}^{1+n(\chi;\tau)}})$ partition $(0\ ,\ \mu_\tau)$ into disjoint sets, then there is  at least one $1\leq i\leq n(\chi;\tau)$, such that $\tilde \mu^{**}(\chi,\tau)\le \mu_{\chi,\tau}^{i +1}$, and
\begin{equation}\label{ZZZ0-0}
\frac{\mu +\tau c_\mu}{\sqrt{1+\tau\mu c_\mu-\mu^2}}+\frac{\mu(\mu+\tau c_\mu)}{1+\tau\mu c_\mu-\mu^2}{ <} \frac{b-\chi}{\chi},\quad \forall \mu \in { (}\mu_{\chi,\tau}^i \ ,\ \mu_{\chi,\tau}^{i+1}).
\end{equation}
Let $i^*$ be the largest $i$ satisfying   \eqref{ZZZ0-0}.
For every  $\chi<\chi_{\tau}^{*}$, define
\begin{equation}
\label{ZZZ0-1}
 \mu^{*}(\chi;\tau)=\mu_{\chi,\tau}^{i^*+1},
\end{equation}
and
\begin{align}\label{ZZZ1}
 \mu^{**}(\chi;\tau)=\inf\Big\{ & \bar{\mu}\in(\tilde \mu^{**}(\chi;\tau),\mu^{*}(\chi;\tau))\ | \nonumber\\
 &\Big(\frac{\mu +\tau c_\mu}{\sqrt{1+\tau\mu c_\mu-\mu^2}}+\frac{\mu(\mu+\tau c_\mu)}{1+\tau\mu c_\mu-\mu^2} \Big)\leq \frac{b-\chi}{\chi},\  \forall \ \bar \mu\le \mu \le  \mu^{*}(\chi;\tau)   \Big\}.
\end{align}
Thus, it follows from \eqref{ZZZ1} that  $0\leq \mu^{**}(\chi;\tau)<\mu^{*}(\chi;\tau)\leq \mu_\tau$ for every $0<\chi<\chi^*_\tau$.
Finally, for every $0<\chi<\chi^*_{\tau}$, we set
$$
c^{*}(\chi;\tau)=c_{\mu^{*}(\chi;\tau)},
$$
and
$$
c^{**}(\chi;\tau):= c_{\mu^{**}(\chi;\tau)}.
$$

\medskip

We claim now that for every $ 0<\chi<\chi^{*}_{\tau}$, and $c\in (c^*(\chi;\tau)\ ,\ c^{**}(\chi;\tau))$, \eqref{main-eq1} has a traveling wave solution connecting $(\frac{a}{b},\frac{a}{b})$ and $(0,0)$ with speed $c$. Indeed, let   $c\in {\rd (}c^*(\chi;\tau)\ ,\ c^{**}(\chi;\tau))$ be given. Since the function $(0\ ,\ \mu_\tau]\ni\mu\mapsto c_{\mu} $ is strictly decreasing, there is a unique $\mu^{**}(\chi;\tau)<\mu(c) <\mu^*(\chi;\tau)$ such that $c=c_{\mu(c)}$. Observe that $\mu(c)$ and $\chi$ satisfy \eqref{Eq01_Th1}. Thus it follows from Theorem \ref{existence-tv-thm} that \eqref{main-eq1} has a traveling wave solution
$(u(x,t),v(x,t))=(U(x-c_\mu t),V(x-c_\mu t))$ satisfying
$$
\lim_{x\to-\infty}U(x)=\frac{a}{b} \quad \text{and}\quad
\lim_{x\to\infty}\frac{U(x)}{e^{-\mu x}}=1.
$$ It is clear from \eqref{ZZZ0}, \eqref{ZZZ0-1}, and \eqref{ZZZ1}  that for every $\tau>0$, we have
$$
\lim_{\chi\to 0+}\mu^*(\chi;\tau)=\mu_{\tau}\quad \text{and} \quad \lim_{\chi\to 0+}\mu^{**}(\chi;\tau)=0.
$$
It  then follows that
$$
\lim_{\chi\to 0+}c^*(\chi;\tau)=\begin{cases}
2\sqrt{a}\qquad \qquad \qquad \text{if}\ 0<a\leq \frac{1+\tau a}{(1-\tau)_+},\\
\frac{1+\tau a}{(1-\tau)_+}+\frac{a(1-\tau)_{+}}{1+\tau a}\quad \text{if}\ a\geq \frac{1+\tau a}{(1-\tau)_+}
\end{cases} \quad \text{and} \quad \lim_{\chi\to 0+}c^{**}(\chi;\tau)=\infty.
$$

 Finally, we consider the limits of $\chi_\tau^*$,   $c^*(\chi;\tau)$,  and $c^{**}(\chi;\tau)$  as $\tau\to 0+$.
From the definition of $m_\tau$, we have that
$$ 1\leq \liminf_{\tau\to 0+}m_{\tau}.$$
On the other hand, we have that
$$
\limsup_{\tau\to 0+}m_\tau\leq \inf\{ \max\{ 1, \  \Big(\frac{\mu}{\sqrt{1-\mu^2}}+\frac{\mu^2}{1-\mu^2}\Big)\}\ | \ \mu\in(0\ ,\  \min\{\sqrt{a}\ ,\ 1\} \}=1.
$$
Hence we have that $\lim_{\tau\to 0+}m_{\tau}=1$, which implies that $\lim_{\tau\to 0+}\chi_{\tau}^*=\frac{b}{2}$.

For every $0<\chi<\frac{b}{2}$,  let   $\mu^{*}(\chi)$  be  given by
$$\mu^*(\chi):=\sup\{\mu\in(0, \min\{\sqrt{a}, 1\}) \ \ |\ \ \Big(\frac{\mu}{\sqrt{1-\mu^2}}+\frac{\mu^2}{1-\mu^2}\Big)\leq \frac{b-\chi}{\chi}\}. $$
Note that $\mu^{*}(\chi)<1$ for every $0<\chi<\frac{b}{2}$. Let $0<\chi<\frac{b}{2}$ be given.
Using the fact that the function $(0, \min\{\sqrt{a},1\})\ni \mu \mapsto \frac{\mu}{\sqrt{1-\mu^2}}+\frac{\mu^2}{1-\mu^2} $ is increasing, we obtain that
\begin{align*}
\lim_{\tau\to 0^+}\Big(\frac{\mu +\tau c_\mu}{\sqrt{1+\tau\mu c_\mu-\mu^2}}+\frac{\mu(\mu+\tau c_\mu)}{1+\tau\mu c_\mu-\mu^2}\Big)& =\Big(\frac{\mu}{\sqrt{1-\mu^2}}+\frac{\mu^2}{1-\mu^2}\Big)\\
&< \Big(\frac{\mu^*(\chi)}{\sqrt{1-(\mu^{*}(\chi))^2}}+\frac{(\mu^{*}(\chi))^2}{1-(\mu^{*}(\chi))2}\Big) \nonumber\\
&\leq \frac{b-\chi}{\chi}\quad {\rm for}\quad \mu\in (0,\mu^*(\chi)),
\end{align*}
 where the limit is uniform in any closed interval of $(0,\min\{\sqrt a,1\})$. It then follows that
$$
\lim_{\tau\to  0+}\mu^{**}(\chi;\tau)=0,\quad \lim_{\tau \to 0+}\mu^*(\chi;\tau)=\mu^*(\chi).
$$
This implies that
$$
\lim_{\tau\to 0+} c^*(\chi;\tau)=c^*(\chi),\quad \lim_{\tau\to 0+}c^{**}(\chi;\tau)=\infty.
$$
\end{proof}

\begin{rk}
\label{rk-thmc}
If $0<a<\frac{1+\tau a}{(1-\tau)_{+}}$, then $\mu_{\tau}=\sqrt a$. If $0<a<\frac{1+\tau a}{(1-\tau)_{+}}$ and
$$
0<\chi< \frac{b}{1+\max\{1+2\tau\sqrt{a}\ ;\ \frac{\sqrt{a}(1+2\tau)}{\sqrt{1+2\tau a-a}} + \frac{a(1+2)}{1+2\tau a-a}\} },
$$
then
 \begin{equation}\label{1}
 \max\{1+\tau c_{\mu_{\tau}}, \frac{\mu_\tau+\tau c_{\mu_{\tau}}}{\sqrt{1+\tau \mu_{\tau}c_{\mu_{\tau}}-\mu^2_\tau}} + \frac{\mu_{\tau}(\mu_\tau+\tau c_{\mu_{\tau}})}{1+\tau \mu_{\tau}c_{\mu_{\tau}}-\mu^2_\tau}\}<\frac{b-\chi}{\chi},
\end{equation}
and
hence $ \mu^{*}(\chi;\tau)=\mu_\tau$ and  $c^*(\chi;\tau)=2\sqrt a$.
\end{rk}

\subsection{Proof of Theorem D}

In this subsection, we prove Theorem D.
To do so, we first prove the following two  lemmas.

\begin{lem}
\label{nonexistence-lm1}
\begin{itemize}
\item[(1)]
Let $0\le c<2\sqrt a$ be  fixed  and $\lambda_0>0$ be such that $c^2-4 a+4\lambda_0<0$. Let $\lambda_D(L)$ be the principal eigenvalue of
\begin{equation}
\label{ev-eq0}
\begin{cases}
\phi_{xx}+c \phi_x +a\phi =\lambda \phi,\quad 0<x<L\cr
\phi(0)=\phi(L)=0.
\end{cases}
\end{equation}
Then there is $L>0$ such that  $\lambda_D(L)=\lambda_0$.

\item[(2)] Let $c$ and $L$ be as in (1). Let $\lambda_D(L;b_1,b_2)$ be the principal eigenvalue of
\begin{equation}
\label{ev-eq1}
\begin{cases}
\phi_{xx}+(c + b_1(x))\phi_x +(a+ b_2(x))\phi =\lambda \phi,\quad 0<x<L\cr
\phi(0)=\phi(L)=0,
\end{cases}
\end{equation}
where $b_1(x)$ and $b_2(x)$ are continuous functions.
 If there is a $C^2$ function $\phi(x)$ with $\phi(x)>0$ for
$0<x<L$  such that
\begin{equation}
\label{ev-eq2}
\begin{cases}
\phi_{xx}+(c + b_1(x))\phi_x +(a+ b_2(x))\phi  \le 0,\quad 0<x<L\cr
\phi(0)\ge 0 ,\quad \phi(L)\ge 0
\end{cases}
\end{equation}
 Then $\lambda_D(L,b_1,b_2)\le 0$.
\end{itemize}
\end{lem}

\begin{proof}
(1)  Let $L>0$ be such that
$$
\frac{\sqrt{4a-4\lambda_0-c^2}}{2}L=\pi.
$$
Then $\lambda=\lambda_0$ is the principal eigenvalue of \eqref{ev-eq0} and $\phi(x)=e^{-\frac{c}{2}x}\sin\Big(\frac{\sqrt {4a-4\lambda_0 -c^2}}{2}x\Big)$
is a corresponding positive eigenfunction.  Hence
$\lambda_D(L)= \lambda_0$ and (1) follows.

(2)  Consider
\begin{equation}
\label{aux-ev-eq}
\begin{cases}
u_t=u_{xx}+(c+b_1(x))u_x+(a+b_2(x))u,\quad 0<x<L\cr
u(x,0)=u(x,L)=0.
\end{cases}
\end{equation}
Let $u(x,t;u_0,b_1,b_2)$ be the solution of \eqref{aux-ev-eq} with $u(x,0;u_0,b_1,b_2)=u_0(x)$ for $u_0\in L^2(0,L)$.
Then we have
$$
\lambda_D(L;b_1,b_2)=\lim_{t\to\infty}\frac{\ln\|u(\cdot,t;u_0,b_1,b_2)\|_{L^2}}{t}
$$
for any $u_0\in L^2(0,L)$ with $u_0\ge 0$ and $u_0\not =0$.
By comparison principle for parabolic equations, $u(x,t;\phi,b_1,b_2)\le \phi(x)$ for all $t\ge 0$ and $0<x<L$. It then follows that
$$
\lambda_D(L;b_1,b_2)\le 0.
$$
\end{proof}

\begin{lem}
\label{nonexistence-lm2}
\begin{itemize}
\item[(1)] Let $c<0$ be fixed and let $\lambda_0>0$ be such that $0<\lambda_0<a$.
Let $\lambda_{N,D}(L)$ be the principal eigenvalue of
\begin{equation}
\label{ev-eq3}
\begin{cases}
\phi_{xx}+c \phi_x +a\phi =\lambda \phi,\quad 0<x<L\cr
\phi_x(0)=\phi(L)=0.
\end{cases}
\end{equation}
Then there is $L>0$ such that  $\lambda_{N,D}(L)=\lambda_0$.

\item[(2)] Let $c$ and $L$ be as in (1). Let $\lambda_{N,D}(L;b_1,b_2)$ be the principal eigenvalue of
\begin{equation}
\label{ev-eq4}
\begin{cases}
\phi_{xx}+(c + b_1(x))\phi_x +(a+ b_2(x))\phi =\lambda \phi,\quad 0<x<L\cr
\phi_x(0)=\phi(L)=0,
\end{cases}
\end{equation}
where $b_1(x)$ and $b_2(x)$ are continuous functions.
 If there is a $C^2$ function $\phi(x)$ with $\phi(x)>0$ for
$0<x<L$  such that
\begin{equation}
\label{ev-eq5}
\begin{cases}
\phi_{xx}+(c + b_1(x))\phi_x +(a+ b_2(x))\phi  \le 0,\quad 0<x<L\cr
\phi_x(0)\le  0 ,\quad \phi(L)\ge 0
\end{cases}
\end{equation}
 Then $\lambda_{N,D}(L,b_1,b_2)\le 0$.
\end{itemize}
\end{lem}

\begin{proof}
(1) Fix $c<0$ and $0<\lambda_0<a$  with $4a-4\lambda_0<c^2$. Let
$$
L=\frac{1}{\sqrt {c^2-4a+4\lambda_0}}\ln \frac{-c+\sqrt {c^2-4a+4\lambda_0}}{-c-\sqrt {c^2-4a+4\lambda_0}}.
$$
 Then $L>0$, $\lambda_{N,D}(L)=\lambda_0$ is the principal eigenvalue of \eqref{ev-eq3},
 and  $\phi(x)$ is
a corresponding  positive eigenfunction, where
$$
\phi(x)=  - e^{\frac{-c+\sqrt {c^2-4a+4\lambda_0}}{2}x} + \frac{-c+\sqrt {c^2-4a+4\lambda_0}}{-c-\sqrt {c^2-4a+4\lambda_0}}  e^{\frac{-c-\sqrt {c^2-4a+4\lambda_0}}{2}x}.
$$
(1) then follows.

(2) It can be proved by the similar arguments as those in Lemma \ref{nonexistence-lm1}(2).
\end{proof}

\begin{proof}[Proof of Theorem D]
We first consider the case that $0\le c<2\sqrt a$. Then there is $\lambda_0>0$ such that
$$
c^2-4 a+4\lambda_0<0.
$$
By Lemma \ref{nonexistence-lm1}(1), there is $L>0$ such that  $\lambda_D(L)=\lambda_0>0$.

Fix $0\le c<2\sqrt a$ and the above $L$.  Assume that   \eqref{main-eq1} has a traveling wave solution
$(u,v)=(U(x-ct),V(x-ct))$ with $(U(-\infty)$, $V(-\infty))=(a/b,a/b)$ and
$(U(\infty),V(\infty))=(0,0)$.  Then \eqref{stationary-eq} has a stationary solution
$(u,v)=(U(x),V(x))$ with $(U(-\infty),V(-\infty))=(a/b,a/b)$ and  $(U(\infty),V(\infty))=(0,0)$.
Moreover,  for any $\epsilon>0$, this is $x_\epsilon>0$ such that
$$
0<U(x)<\epsilon, \quad 0<V(x)<\epsilon,\quad |V_x(x)|<\epsilon \quad \forall \,\, x\ge x_\epsilon.
$$

Consider the eigenvalue problem,
\begin{equation}
\label{ev-eq6}
\begin{cases}
\phi_{xx}+(c-\chi V_x(x))\phi_x+(a-\chi( V(x)-\tau cV_{x})-(b-\chi)U(x))\phi=\lambda\phi,\quad x_\epsilon<x<x_\epsilon+L\cr
\phi(x_\epsilon)=\phi(x_\epsilon+L)=0.
\end{cases}
\end{equation}
Let $\lambda_D^\epsilon(L)$ be the principal eigenvalue of \eqref{ev-eq6}.
By Lemma \ref{nonexistence-lm1}(1) and perturbation theory for principal eigenvalues of elliptic operators,  $\lambda_D^\epsilon(L)>0$ for $0<\epsilon\ll 1$.

Note that
$$
U_{xx}+(c -\chi V_{x})U_{x} + (a-\chi( V(x)-\tau cV_{x})-(b-\chi)U(x))U=0\quad \forall\,\, x_\epsilon\le x\le x_\epsilon+L
$$
and $U(x_\epsilon)>0$, $U(x_\epsilon+L)>0$. Then, by Lemma \ref{nonexistence-lm1}(2),  $\lambda_D^\epsilon(L)\le 0$. We get a contradiction.
Therefore,  \eqref{main-eq1} has no traveling wave solution
$(u,v)=(U(x-ct),V(x-ct))$ with $(U(-\infty),V(-\infty))=(a/b,a/b)$ and  $(U(\infty),V(\infty))=(0,0)$ and $0\le c<2\sqrt a$.

\smallskip
Next, we consider the case that $c<0$. Let $\lambda_0$ and $L$ be as in Lemma \ref{nonexistence-lm2}(1).
Then $\lambda_{N,D}(L)=\lambda_0>0$.

Fix $c<0$ and the above $L$.
 Assume that   \eqref{main-eq1} has a traveling wave solution
$(u,v)=(U(x-ct),V(x-ct))$ with $(U(-\infty)$, $V(-\infty))=(a/b,a/b)$ and
$(U(\infty),V(\infty))=(0,0)$.
Then \eqref{stationary-eq} has a stationary solution
$(u,v)=(U(x),V(x))$ with $(U(-\infty),V(-\infty))=(a/b,a/b)$ and  $(U(\infty),V(\infty))=(0,0)$.
Similarly, for any $\epsilon>0$, this is $x_\epsilon>0$ such that
$$
0<U(x)<\epsilon, \quad 0<V(x)<\epsilon,\quad |V_x(x)|<\epsilon \quad \forall \,\, x\ge x_\epsilon.
$$
Moreover,
since $U(\infty)=0$, there is $\tilde x_\epsilon>x_\epsilon$ such that
$$
U_x(\tilde x_\epsilon)<0.
$$

Consider the eigenvalue problem,
\begin{equation}
\label{ev-eq7}
\begin{cases}
\phi_{xx}+(c-\chi V_x(x))\phi_x+(a-\chi( V(x)-\tau cV_{x})-(b-\chi)U(x))\phi=\lambda\phi,\quad \tilde x_\epsilon<x<\tilde x_\epsilon+L\cr
\phi_x(\tilde x_\epsilon)=\phi(\tilde x_\epsilon+L)=0.
\end{cases}
\end{equation}
 Let $\lambda_{N,D}^\epsilon(L)$ be the principal eigenvalue of \eqref{ev-eq7}.
By Lemma \ref{nonexistence-lm2}(1) and perturbation theory for principal eigenvalues of elliptic operators,  $\lambda_{N,D}^\epsilon(L)>0$ for $0<\epsilon\ll 1$.

Note that
$$
U_{xx}+(c -\chi V_{x})U_{x} + (a-\chi( V(x)-\tau cV_{x})-(b-\chi)U(x))U=0\quad \forall\,\, \tilde x_\epsilon\le x\le\tilde  x_\epsilon+L
$$
and $U_x(\tilde x_\epsilon)<0$, $U(\tilde x_\epsilon+L)>0$. Then, by Lemma \ref{nonexistence-lm2}(2),  $\lambda_{N,D}^\epsilon(L)\le 0$. We get a contradiction.
Therefore,  \eqref{main-eq1} has no traveling wave solution
$(u,v)=(U(x-ct),V(x-ct))$ with $(U(-\infty),V(-\infty))=(a/b,a/b)$ and  $(U(\infty),V(\infty))=(0,0)$ and $c<0$.

Theorem D is thus proved.
\end{proof}

\end{document}